\documentclass[]{amsart}
\usepackage[all]{xy}
\usepackage{enumerate}
\usepackage{amssymb}

\newtheorem{thm}{Theorem}[section]
\newtheorem{prop}[thm]{Proposition}
\newtheorem{lem}[thm]{Lemma}
\newtheorem{cor}[thm]{Corollary}

\theoremstyle{definition}
\newtheorem{defn}[thm]{Definition}
\newtheorem{examp}[thm]{Example}

\theoremstyle{remark}
\newtheorem{rmk}[thm]{Remark}

\newcommand{\rmap}{\longrightarrow}
\newcommand{\lmap}{\longleftarrow}

\newcommand{\pr}         {{\mathrm{pr}}}

\newcommand{\tto}{\rightrightarrows}    
\renewcommand{\gg}{\mathfrak g}
\newcommand{\can}{\mathrm{can}}

\DeclareMathOperator{\Ker}        {ker}

\DeclareMathOperator{\Aut}        {Aut}
\DeclareMathOperator{\OutAut}     {OutAut}
\DeclareMathOperator{\OutGaug}     {OutGaug}
\DeclareMathOperator{\InnAut}     {InnAut}
\DeclareMathOperator{\Pic}        {Pic}

\DeclareMathOperator{\Ad}         {Ad}
\DeclareMathOperator{\GL}         {GL}

\DeclareMathOperator{\Diff}       {Diff}

\DeclareMathOperator{\Bis}        {Bis}
\DeclareMathOperator{\LBis}       {LBis}
\DeclareMathOperator{\IBis}       {IsoBis}
\DeclareMathOperator{\IsoL}       {IsoLBis}
\newcommand{\Rr}{\mathbb R}
\newcommand{\Ss}{\mathbb S}
\newcommand{\Zz}{\mathbb Z}

\newcommand{\Tt}{\mathbb T}
\renewcommand{\d}{\mathrm d}

\newcommand{\eps}{\varepsilon}

\newcommand{\X}{\ensuremath{\mathfrak{X}}}

\newcommand{\Ham}{\text{Ham}}
\newcommand{\bas}{\mathrm{bas}}
\newcommand{\cl}{\mathrm{cl}}

\newcommand{\G}{\mathcal{G}}            
\newcommand{\s}{\mathbf{s}}             
\renewcommand{\t}{\mathbf{t}}           

\renewcommand{\d}{\mathrm d}               
\newcommand{\Lie}{\boldsymbol{\pounds}}    

\newcommand{\al}{\alpha}
\newcommand{\be}{\beta}

\newcommand{\frakg}     {\mathfrak{g}}

\newcommand{\pic}       {\mathfrak{pic}}

\numberwithin{equation}{section}

\begin{document}

\title{\bf Picard groups of Poisson manifolds}

\author{Henrique Bursztyn}
\address{IMPA, Estrada Dona Castorina 110, Jardim Botanico, Rio de Janeiro, 22460-320, Brazil}
\email{henrique@impa.br}

\author{Rui Loja Fernandes}
\address{Department of Mathematics,
University of Illinois at Urbana-Champaign,
1409 W. Green Street,
Urbana, IL 61801
USA }
\email{ruiloja@illinois.edu}

\thanks{HB has had the support of Faperj, and RLF is partially supported
by NSF grants DMS 1308472 and DMS 1405671. Both authors acknowledge the support
of a Capes/Brazil-FCT/Portugal cooperation grant and the
\emph{Ci\^encias Sem Fronteiras} program sponsored by CNPq.}

\date{\today}


\begin{abstract}
For a Poisson manifold $M$ we develop systematic methods to compute
its Picard group $\Pic(M)$, i.e., its group of self Morita
equivalences. We establish a precise relationship between $\Pic(M)$
and the group of gauge transformations up to Poisson diffeomorphisms
showing, in particular, that their connected components of the
identity coincide; this allows us to introduce the Picard Lie
algebra of $M$ and to study its basic properties. Our methods lead
to the proof of a conjecture from \cite{BW04} stating that
$\Pic(\gg^*)$ for any compact simple Lie algebra agrees with the
group of outer automorphisms of $\gg$.
\end{abstract}

\maketitle


\tableofcontents

\section{Introduction}

The Picard group of an (integrable) Poisson manifold was introduced
by A. Weinstein and the first author in \cite{BW04}, as an analogue
of the notion of Picard group in ring theory. Picard groups may be
seen as groups of automorphisms once the usual notion of isomorphism
is replaced by {\em Morita equivalence}, a weaker relation that
identifies objects with equivalent categories of modules. In other
words, the Picard group of a Poisson manifold is defined, just as in
ring theory, as the group of {\em self Morita equivalences}.

The notion of Morita equivalence in Poisson geometry goes back  to
the work of P. Xu \cite{Xu91}. In this context, the role of
``module'' is played by {\em symplectic realizations}, while Morita
equivalence is defined in terms of {\em dual pairs} of Poisson
manifolds, in the sense of \cite{We83}. As expected, Morita
equivalent Poisson manifolds share many key properties: for example,
they have homeomorphic leaf spaces, and the transverse geometry of
corresponding symplectic leaves is the same.
From a geometric point of view, one may think of Morita equivalence
as identifying Poisson manifolds modulo the ``internal'' symplectic
geometry of the leaves, so Picard groups do not encode this
symplectic information -- though they codify the topology of the leaves (e.g., their fundamental groups)
and the transversal variation of leafwise symplectic forms. For example, the Picard group of a
symplectic manifold $M$ coincides with the group of outer
automorphisms of $\pi_1(M)$ (see \cite{BW04}).

Although the Picard group $\Pic(M)$ of a Poisson manifold $(M,\pi)$
is a natural invariant, computations are usually very hard. In fact,
up to now, Picard groups have been described only in a handful of
examples, treated in a case-by-case basis. Our aim in this paper is
to relate the Picard group with more computable groups associated
with a Poisson manifold; this allows us to develop more systematic
methods to calculate $\Pic(M)$ while gaining further geometric
insight. In particular, we identify a large (open) subgroup of the
Picard group which can be described geometrically and quite
explicitly. In favorable circumstances, this group actually
coincides with the Picard group and, hence, can be used as a
computational tool. As a byproduct, we obtain a way to identify the
infinitesimal counterpart of the Picard group, the Picard Lie
algebra $\pic(M)$, which was not previously known.

Let us describe in more detail the main results of this paper. After
recalling basic definitions and introducing the group $\Pic(M)$, we
discuss examples of Poisson manifolds illustrating that $\Pic(M)$
can range from being a finite group to being very ``large''.
In order to gain insight into this issue, we consider the group of
\emph{gauge transformations up to Poisson diffeomorphism},
\[
\G_{\pi}(M):=\{(\phi,B)\in
\Diff(M)\ltimes\Omega_{\cl}^2(M)~|~\phi_*\pi_B=\pi\},
\]
where $\pi_B:=e^B\pi$ denotes the $B$-transform of $\pi$, as in
\cite{SeWe} (so $\pi_B$ has the same foliation as $\pi$ but the
leafwise symplectic forms differ by the pullback of $B$); the group
operation in $\G_{\pi}(M)$ is that of semi-direct product. 

Our first main result can be stated as follows:

\begin{thm}
\label{thm:main:introd} There is an exact sequence of groups,
\[
\xymatrix{ 1\ar[r] &\IsoL(\Sigma(M))\ar[r]&
\Bis(\Sigma(M))\ar[r]&\G_\pi(M)\ar[r]&\Pic(M),}
\]
such that the image of the last arrow is the normal subgroup of $\Pic(M)$
formed by the isomorphism classes of self Morita bimodules which admit a bisection.
\end{thm}
In this exact sequence, $\Bis(\Sigma(M))$ denotes the group of
bisections of $\Sigma(M)$, the source-simply-connected symplectic
groupoid integrating $M$, while $\IsoL(\Sigma(M))$ is its subgroup
formed by Lagrangian bisections which take values in the isotropy, see Section~\ref{sec:picard}.

We also consider a natural topology on $\Pic(M)$ and verify that the
image of the last map in the exact sequence above is open and a
topological group. If $\Sigma(M)$ is compact, we show that $\Pic(M)$
is itself a topological group (even a Lie group, although usually
infinite dimensional), and we believe that this should also hold in
the non-compact case. Note that the other groups in the exact
sequence of Theorem \ref{thm:main:introd} are spaces of maps and
hence carry natural $C^\infty$-topologies.

Since the image of the last arrow is open in $\Pic(M)$, it defines a
``large" subgroup of $\Pic(M)$, which contains the connected
component of the identity. This allows us to realize the Lie algebra
of the Picard group, denoted $\pic(M)$, as a quotient of
$\gg_\pi(M)$, the Lie algebra of $\G_\pi(M)$. As a result, we obtain
the following infinitesimal counterpart of Theorem
\ref{thm:main:introd}:

\begin{thm}
\label{thm:main:algebra:introd} The Picard Lie algebra fits into an
exact sequence of Lie algebras as follow:
\[
\xymatrix{ 0\ar[r] &\Omega^1_{\cl,\bas}(M)\ar[r]&
\Omega^1(M)\ar[r]&\frakg_\pi(M)\ar[r]&\pic(M)\ar[r]&0.}
\]
\end{thm}

In this last exact sequence, $\Omega^1(M)$ denotes the Lie algebra
of 1-forms with the Koszul Lie bracket $[~,~]_\pi$ induced by the
Poisson tensor, and $\Omega^1_{\cl,\bas}(M)$ is its subalgebra
consisting of closed, basic (relative to the symplectic foliation)
1-forms. The Lie algebra $\gg_\pi(M)$ is also given explicitly as
\[
\gg_\pi:=\{(Z,\beta)\in\X(M)\ltimes \Omega^2_{\cl}(M)~:~\d\beta=0,\
\Lie_Z\pi=\pi^\sharp(\beta)\},
\]
viewed as a subalgebra of the semi-direct product $\X(M)\ltimes
\Omega^2_{\cl}(M)$ (where the Lie algebra of vector fields acts on
the abelian Lie algebra of closed 2-forms by Lie derivative). This
leads to the following characterization of the Picard Lie algebra,
which reveals its connection with Poisson cohomology:

\begin{cor}
\label{cor:Picard:Lie:algebra} The Picard Lie algebra is given by
 \[
 \pic(M)=H^1(\pi^\sharp),
 \]
the 1st relative cohomology group associated with the morphism of
complexes
\[
\pi^\sharp:(\Omega^\bullet(M),\d_{\mathrm{dR}})\to
(\X^\bullet(M),\d_\pi).
\]
In particular, it fits into a long exact sequence:
\[
\xymatrix{ \cdots \ar[r] & H^1(M)\ar[r] & H^1_\pi(M)\ar[r]
&\pic(M)\ar[r] & H^2(M)\ar[r] & H^2_\pi(M)\ar[r] &\cdots} \]
\end{cor}

These results can be explored in many situations to compute Picard
groups/Lie algebras, or at least to gain insight into their
structure. For example, from the exact sequence in
Theorem~\ref{thm:main:introd}, one obtains conditions for the Picard
group to coincide with the group $\OutAut(M)$ of \emph{outer Poisson
diffeomorphisms} of $(M,\pi)$: e.g., this happens if all symplectic
bimodules over $M$ admit a Lagrangian bisection (see
Corollary~\ref{cor:Picard:Lagrangian:bisections}).

Deciding whether a bimodule admits a (Lagrangian) bisection is, in
general, a difficult problem. Even when $M$ is a symplectic
manifold, this is a nontrivial issue closely related to the Nielsen
realization problem: in this case, the Picard group
coincides with the group of outer automorphisms of $\pi_1(M)$, and
we will show that a bimodule admits a (Lagrangian) bisection if and
only if the corresponding outer automorphism can be represented by a
(symplectic) diffeomorphism. This issue already played a central
role in the computations in \cite{RS} of the Picard group of log
symplectic structures on compact surfaces.

Another consequence of Corollary \ref{cor:Picard:Lie:algebra} is the
fact that, if $H^2(M)=0$, then
\begin{equation}\label{eq:firstp}
\pic(M)\simeq H^1_\pi(M)/\pi^\sharp(H^1(M)).
\end{equation}
In particular, if $H^2(M)=H^1(M)=0$ then the Picard Lie algebra
$\pic(M)$ is isomorphic to the first Poisson cohomology
$H^1_\pi(M)$, which represents the outer derivations of $(M,\pi)$, i.e., the Lie algebra of  $\OutAut(M)$.

We will apply the results above to recover many of the \emph{ad hoc}
computations of the Picard group that one can find in the literature
and to compute the Picard group (or the Picard Lie algebra) in new
situations.

For example, the Picard group of the Poisson structure
on the dual of a Lie algebra $\gg^*$ has never been calculated before. As an immediate consequence of \eqref{eq:firstp}, we have that,
for any Lie algebra $\gg$,
\[
\pic(\gg^*)\simeq H^1_\pi(\gg^*).
\]
In the case of a semisimple Lie algebra
of compact type we conclude that $\Pic(\gg^*)$ is discrete, because
of the well-known fact that $H^1_\pi(\gg^*)=0$ (see, e.g.,
\cite{GW92}). Actually, in this case, the Picard group coincides
with the finite group of automorphisms of the Dynkin diagram of
$\gg$, since we can prove the following conjecture of \cite{BW04}:

\begin{thm}
If  $\gg$ is a semi-simple Lie algebra of compact type, then
\[
 \Pic(\gg^*)\simeq\OutAut(\gg).
 \]
\end{thm}

This result is remarkable in that, on the right hand side, one has
the automorphisms of the Dynkin diagram, a combinatorial object,
while on the left side one has an object only determined by the
Poisson geometry.  Our proof consists of two steps: we first show
that $\Pic(\gg^*)\simeq\OutAut(\gg^*)$ using
Theorem~\ref{thm:main:introd} (see
Corollary~\ref{cor:Picard:Lagrangian:bisections}); then we check
that $\OutAut(\gg^*)\simeq\OutAut(\gg)$ using a Moser-type trick. We
conjecture that for a general Lie algebra the first isomorphism
still holds.

\vskip 10 pt

\noindent \textbf{Acknowledgments.} We are thankful to several
institutions for their hospitality during various stages of this
project, including IST and UIUC (H.B.) as well as IMPA (R.L.F.). We
would like to thank Gustavo Granja for pointing out the construction
in Example \ref{ex:Gustavo}. We also thank Marius Crainic, David Li-Bland, David
Martinez Torres, Eckhard Meinrenken and Alan Weinstein for valuable discussions and
comments on the paper.

\section{Symplectic realizations and Morita equivalence}
\label{sec:basic:concepts}

We denote by $(M,\{~,~\})$ a Poisson manifold with associated
Poisson bivector field $\pi\in\X^2(M)$. The corresponding bundle map
is denoted by $\pi^\sharp:T^*M\to TM$, so that Hamiltonian vector
fields are given by
\[ X_f:=\{f,\cdot\}=\pi^\sharp\d f. \]
The space of Hamiltonian vector fields is denoted by $\X_\Ham(M)$.
In this section, we review the necessary background on Poisson
geometry.

\subsection{Symplectic realizations}

One way to unravel the complicated geometry of a Poisson manifold
$(M,\pi)$ is to exhibit $M$ as a quotient of a symplectic manifold.
This leads to the following definition.

\begin{defn}
A \textbf{symplectic realization} of $(M,\pi)$ is a symplectic
manifold $(S,\omega)$ together with a surjective submersion $J:S\to
M$ which is also a Poisson map: $J_*\omega^{-1}=\pi$. We say that
two symplectic  realizations $J_i: S_i\to M$, $i=1,2$, are
\textbf{isomorphic} if there exists a symplectomorphism $\varphi:
S_1\to S_2$ such that $J_1 = J_2\circ \varphi$.
\end{defn}

Symplectic realizations were introduced by Weinstein in the
foundational paper \cite{We83}, where the next result is proven.

\begin{thm}[Weinstein \cite{We83}]
Any Poisson manifold admits a Hausdorff symplectic realization.
\end{thm}

One recent version of the proof of this result (see \cite{CrMr11})
goes as follows. Recall (see \cite{Vai94}) that a contravariant
connection $\nabla:\Omega^1(M)\times\Omega^1(M)\to\Omega^1(M)$,
written $(\al,\be)\mapsto\nabla_\al\beta$, is a $\Rr$-bilinear map
satisfying
\[
\nabla_{f\al}\be=f\nabla_\al\beta,\quad
\nabla_\al(f\beta)=f\nabla_\al\beta+\Lie_{\pi^\sharp\al}(f)\beta.\]
For such connections, one can define parallel transport along
cotangent paths, i.e., paths $a:I\to T^*M$ such that
\[ \frac{\d}{\d t}p(a(t))=\pi^\sharp(a(t)),\quad \forall t\in I.\]
Here $I=[0,1]$ denotes the unit interval, and $p:T^*M\to M$ is the
bundle projection. In particular, a geodesic of $\nabla$ is a
cotangent path $a:I\to T^*M$ such that $\nabla_{a(t)}a(t)=0$, for
all $t\in I$. Just as for usual connections, one can define the
geodesic flow, which now is a 1-parameter group of (locally defined)
diffeomorphisms $\Phi^t:T^*M\to T^*M$. For more details on these
constructions we refer to \cite{CrFe11}.

In order to construct a symplectic realization of $(M,\pi)$, we
choose any contravariant connection $\nabla$ and consider its geodesic flow
$\Phi^t$. If $\Omega_\can$ is the canonical symplectic form on
$T^*M$, then the form
\[ \omega=\int_0^1(\Phi^t)^*\Omega_\can\ \d t\]
is well defined and symplectic on a neighborhood $S$ of the zero
section in $T^*M$. The restriction of the canonical projection then
gives the desired symplectic realization $p:(S,\omega)\to (M,\pi)$
(see \cite{CrMr11} for details).

\subsection{Complete symplectic realizations}

The fibers of a symplectic realization are typically not compact.
Requiring compactness is usually too strong a condition. In fact, in
Poisson geometry one replaces the notion of a proper map by the
notion of a complete map: a Poisson map $\phi:(M,\pi_M)\to(N,\pi_N)$
is called \textbf{complete} if whenever $X_h\in\X_\Ham(N)$ is a
complete vector field, it follows that the vector field
$X_{h\circ\phi}\in\X_\Ham(M)$ is also complete. Complete symplectic
realizations do not always exist, as shown by the following result.

\begin{thm}[Crainic \& Fernandes \cite{CrFe04}]
\label{thm:complete:realz}
A Poisson manifold admits a complete symplectic realization if and only if it is integrable.
\end{thm}

Since some of the ideas behind this result will be useful in the
sequel, we make a small digression into the notion of integrability,
following the approach due to Cattaneo-Felder \cite{CaFe01} and
Crainic-Fernandes \cite{CrFe04,CrFe03}.

First of all, the cotangent bundle $T^*M$ of any Poisson manifold
carries a Lie algebroid structure with anchor $\pi^\sharp:T^*M\to
TM$ and Lie bracket
$[~,~]_\pi:\Omega^1(M)\times\Omega^1(M)\to\Omega^1(M)$ given by the
Koszul bracket
\begin{equation}
\label{eq:Koszul:bracket}
[\eta_1,\eta_2]_\pi=\Lie_{\pi^\sharp\eta_1}\eta_2-\Lie_{\pi^\sharp\eta_2}\eta_1-\d(\pi(\eta_1,\eta_2)).
\end{equation}
For a Poisson manifold $M$ we denote by $\Sigma(M)=\G(T^*M)$ its
canonical integration:
\[
\Sigma(M)=\frac{\text{cotangent paths}}{\text{cotangent
homotopies}}.
\]
We recall that $\Sigma(M)$ is a \emph{topological groupoid} with
simply connected source-fibers. It is a (infinite dimensional)
symplectic quotient of the space of all paths in cotangent bundle
$P(T^*M)\simeq T^* P(M)$. Notice that our groupoids need not be Hausdorff (although the base and source/target fibers are always assumed to be Hausdorff).

A Poisson manifold $M$ is said to be \textbf{integrable} if the
associated Lie algebroid $T^*M$ is integrable, i.e., it arises from a \emph{Lie groupoid}.
This happens if and only if $\Sigma(M)\tto M$ is a Lie groupoid, i.e., the
quotient above is a smooth manifold. In this case, the quotient
symplectic structure $\Omega$ on $\Sigma(M)$ makes it into a
\textbf{symplectic groupoid}. This means that the symplectic
structure and the multiplication are compatible: if
$m:\Sigma(M)^{(2)}\to\Sigma(M)$ denotes the multiplication defined
on the submanifold $\Sigma(M)^{(2)}\subset \Sigma(M)\times\Sigma(M)$
of composable arrows, then
\begin{equation}
\label{eq:multiplicative} m^*\Omega=\pr_1^*\Omega+\pr_2^*\Omega,
\end{equation}
where $\pr_i: \Sigma(M)^{(2)}\to \Sigma(M)$ are the (restrictions of
the) projections on each factor.

The fact that $(\Sigma(M),\Omega)$ is a symplectic groupoid implies
that:
\begin{enumerate}[(i)]
\item the target map $\t:\Sigma(M)\to M$, $[a]\mapsto p(a(0))$
(respectively, the source map $\s:\Sigma(M)\to M$, $[a]\mapsto
p(a(1))$) is Poisson (respectively, anti-Poisson);
\item the identity section $\eps:M\to\Sigma(M)$, $m\mapsto [0_m]$ is a
Lagrangian embedding;
\item the inverse map $\iota:\Sigma(M)\to\Sigma(M)$, $[a]\mapsto
[a]^{-1}:=[\bar{a}]$ is an anti-symplectic involution (here
$\overline{a}$ denotes the cotangent path $\bar{a}(t):=-a(1-t)$).
\end{enumerate}
Actually, it is not hard to check that the target fibration
$\t:(\Sigma(M),\Omega)\to M$ defines a \emph{complete} symplectic
realization, and this gives (the easy) half of Theorem
\ref{thm:complete:realz}. The more difficult part of the theorem
follows from the fact that complete symplectic realizations can be
thought of as symplectic $\Sigma(M)$-modules:

\begin{thm}[Mikami \& Weinstein \cite{MiWe88}]\label{thm:MW88}
Every complete symplectic realization $p:(S,\omega)\to(M,\pi)$
determines a symplectic action of $\Sigma(M)\tto M$ on $p:S\to M$.
\end{thm}

In fact, let $[a]\in\Sigma(M)$ be represented by a cotangent path
$a:I\to T^*M$ with base path $\gamma(t)$. Given any $u\in S$ such
that $p(u)=\s([a])=\gamma(0)$, there exists a unique path
$\tilde{\gamma}:I\to S$, lifting $\gamma$ and starting at $u$
($\tilde{\gamma}(0)=u$), satisfying
\[
(\d_{\tilde{\gamma}(t)}p)^*a(t)=i_{\dot{\tilde{\gamma}}(t)}\omega.\]
Then $[a]\cdot u:=\tilde{\gamma}(1)$ defines an action of
$\Sigma(M)\tto M$ on $p:S\to M$. Note that completeness guarantees
that the lift $\tilde{\gamma}(t)$ is defined for every $t$ up to
$t=1$. The definition of this action \emph{does not} appeal to the
smooth structure on $\Sigma(M)$. In fact, given a complete
symplectic realization, we can use this observation to identify $\Sigma(M)\times S\tto S$
with the homotopy groupoid of the symplectic orthogonal foliation to the fibers of $p:S\to M$, from which it follows that
$\Sigma(M)$ is a Lie groupoid, proving the second half of Theorem
\ref{thm:complete:realz}.

\subsection{Lagrangian sections}

As we have just observed, the target map
\[
\t:(\Sigma(M),\Omega)\to M
\]
is a complete symplectic realization. In our study of the Picard
group it is important to understand if the converse holds: when is a
complete realization $p:(S,\omega)\to M$ isomorphic to
$\t:(\Sigma(M),\Omega)\to M$? If such an isomorphism exists then (i)
the fibers of $p$ are isomorphic to the fibers of $\t$, hence are
1-connected and (ii) we can transport through this isomorphism the
identity section $\eps:M\to\Sigma(M)$ obtaining a section $b:M\to S$
which is Lagrangian: $b^*\omega=0$. It turns out that these two
necessary conditions are also sufficient:

\begin{thm}[Coste, Dazord \& Weinstein \cite{CDW87}]
\label{thm:canonical:realz} A complete symplectic realization
$p:(S,\omega)\to M$ is isomorphic to $\t:(\Sigma(M),\Omega)\to M$ if
and only if $p:S\to M$ has 1-connected fibers and admits a
Lagrangian section $b:M\to S$. In this case, the isomorphism is
unique.
\end{thm}

The isomorphism $\Phi:\Sigma(M)\to S$ is obtained using the action
in Thm.~\ref{thm:MW88}  by
\[
\Phi([a])=[a]\cdot b(\s([a])).
\]
Clearly this isomorphism takes the identity section
$\eps:M\to\Sigma(M)$ to the Lagrangian section $b$.

The cotangent bundle $p: T^*M\to M$ with its canonical symplectic
form $\Omega_\can$ is a very special example of a symplectic
groupoid, where $\s=\t=p$, and multiplication is fibrewise addition.
Recall that the canonical symplectic form $\Omega_\can$ is
characterized by the following fundamental property: for any 1-form
$\al\in\Omega^1(M)$ one has
\[
\al^*\Omega_\can=\d \al,
\]
where on the left-hand side we view $\al$ as a section of $p:T^*M\to
M$. It turns out that this fundamental property of $\Omega_\can$ has
a version valid for any symplectic groupoid as we now explain.

Let $\G\tto M$ be a Lie groupoid. A \textbf{bisection of $\G$} is an
embedded submanifold $L\subset \G$ such that the restrictions of
both $\s$ and $\t$ to $L$ induce diffeomorphisms $L\to M$. We can
always parameterize a bisection by a map $b:M\to \G$ such that
$\s\circ b=$id and $\t\circ b:M\to M$ is a diffeomorphism. The set
of all bisections $\Bis(\G)$ forms a group under the obvious
composition. The set of {\bf Lagrangian bisections} defines a subgroup
denoted by $\LBis(\G)$.

The exponential map of a Lie groupoid $\G\tto M$ is a map
$\exp_\G:\Gamma(A)\to\Bis(\G)$ which associates to any small enough
section of its Lie algebroid $A\to M$ (e.g., a compactly supported
section) a bisection of the groupoid: if $\al\in \Gamma(A)$ then
$\exp(\al)\in\Bis(\G)$ is the bisection defined by
\[
\exp(\al)(m)=\phi^1_{\tilde{\al}}(1_m),
\]
where $\tilde{\al}$ is the right-invariant vector field in $\G$
defined by the section $\al$, and $\phi^t_{\tilde{\al}}$ denotes the
flow of $\tilde{\al}$. It should be clear from the
definition that
\[ \s\circ\exp(\al)=\text{id, }\quad \t\circ\exp(\al)=\phi^1_{\rho(\al)},\]
where $\rho:A\to TM$ denotes the anchor of $A\to M$.

\begin{prop}
\label{prop:bisections:form} Let $(\G,\Omega)\tto M$ be a symplectic
groupoid. Then its exponential map $\exp:\Omega^1(M)\to\Bis(\G)$
satisfies
\[ \exp(\al)^*\Omega=\d\al, \quad \forall \al\in\Omega^1(M).\]
In particular, $\exp(\al)$ is a Lagrangian bisection if and only if $\al$ is a closed 1-form.
\end{prop}

For the proof we refer to \cite{Xu97}. Notice that when $\G=T^*M$
our definition gives $\exp(\al)=\al$ and the proposition reduces to
the fundamental property of $\Omega_\can$.

\subsection{Morita Equivalence}

We henceforth restrict our attention to integrable Poisson manifolds
$(M,\pi)$ with $\Sigma(M)$ a Hausdorff symplectic Lie groupoid.

Two Poisson manifolds $(M,\pi_M)$ and $(N,\pi_N)$ are said to be
\textbf{Morita equivalent} \cite{Xu91} if there exists a symplectic
manifold $(S,\omega)$ and a two leg diagram
\[
\xymatrix@R=10pt{
 &S \ar[dl]_{p}\ar[dr]^{q}& \\
M& & \overline{N},}
\]
where $p$ and $q$ are complete symplectic realizations with
1-connected fibers such that the sub-bundles tangent to the $p$- and
$q$-fibers are symplectic orthogonal complements of one another
(here, as usual, the bar indicates that we change the sign of the
Poisson bracket.) The orthogonality of the fibers implies, in
particular, that
\begin{equation}
\label{eq:commute}
 \{f\circ p,g\circ q\}_S=0,\quad \forall f\in C^\infty(M),g\in C^\infty(N).
\end{equation}
We shall refer to $S$ as a \textbf{Morita bimodule}.

Two Morita bimodules $M\stackrel{p}{\lmap} S\stackrel{q}{\rmap}
\overline{N}$ and $M\stackrel{p'}{\lmap} S'\stackrel{q'}{\rmap}
\overline{N}$ are said to be \textbf{equivalent} if there is a
symplectic isomorphism $\Phi:(S,\omega)\to (S,\omega')$ which makes
the following diagram commute:
\[
\xymatrix@R=8pt@C=15pt{
& S \ar[dddl]_{p}\ar'[dr][dddrrr]_{q}\ar[rr]^{\Phi}& &S' \ar[dddlll]^{p'}\ar[dddr]^{q'} \\
& & & &\\
\\
M&& & &\overline{N}.}
\]

Note (see Thm.~\ref{thm:MW88}) that a Morita bimodule
$M\stackrel{p}{\lmap} S\stackrel{q}{\rmap} \overline{N}$ gives rise
to left and right symplectic groupoid actions,
\[
\xymatrix{
\Sigma(M)\ar@<3pt>[d]\ar@<-3pt>[d] & S\, \ar@<-8pt>@(ul,dl) \ar[dl]_p \ar[dr]^q \ar@(ur,dr)@<2pt>& \Sigma(\overline{N})\ar@<3pt>[d]\ar@<-3pt>[d]\\
M & & \overline{N}, }
\]
which commute because of \eqref{eq:commute}. Moreover, each action
is principal and the orbit space is determined by the other map, so
$q:S\to N$ (respectively, $p:S\to M$) induces an isomorphism
$S/\Sigma(M)\simeq N$ (respectively, $S/\Sigma(N)\simeq M$).

\begin{examp}[Poisson diffeomorphisms]
\label{ex:Poisson:diffeo}
The symplectic groupoid $\Sigma(M)$ can be viewed as a self Morita equivalence of $M$:
\[
\xymatrix@R=10pt{
 &\Sigma(M) \ar[dl]_{\t}\ar[dr]^{\s}& \\
M& & \overline{M}.}
\]
The corresponding left/right actions are the left/right actions of $\Sigma(M)$ on itself.

More generally, every Poisson diffeomorphism
$\phi:(M,\pi_M)\to(N,\pi_N)$ induces a Morita equivalence
\[
\xymatrix@R=10pt{
 &\Sigma(M) \ar[dl]_{\t}\ar[dr]^{\phi\circ\s}& \\
M& & \overline{N}.}
\]
The left action of $\Sigma(M)$ is still the action by left
translations of $\Sigma(M)$ on itself. For the right action of
$\Sigma(\overline{N})$ on $\Sigma(M)$ one first integrates
$\phi:M\to N$ to a symplectic groupoid isomorphism
$\Phi:\Sigma(M)\to\Sigma(N)$ and then $x\in\Sigma(N)$ acts on
$\Sigma(M)$ by right translation by $\Phi^{-1}(x)$. We will denote
this Morita bimodule by $\Sigma(M)_\phi$.
\end{examp}

\begin{examp}[Gauge transformations]
\label{ex:gauge:transf} Another important class of Morita equivalences is given by gauge equivalences, as observed by
Bursztyn and Radko in \cite{BR03}. Given a Poisson manifold
$(M,\pi)$, we say that a 2-form $B\in\Omega^2(M)$ defines a
\textbf{gauge equivalence} \cite{SeWe} if
\begin{enumerate}[(a)]
\item $\d B=0$ and
\item the bundle map $I+B^\flat\circ\pi^\sharp: T^*M\to T^*M$ is invertible.
\end{enumerate}
For such a 2-form, the bivector $\pi_B\in\X^2(M)$ given by
$(\pi_B)^\sharp=\pi^\sharp\circ(I+B^\flat\circ\pi^\sharp)^{-1}$
defines a new Poisson structure on $M$. The geometric interpretation
of $\pi_B$ is as follows: it has the same foliation as $\pi$ while
the symplectic form on a leaf differs by the restriction of $B$ to
the leaf. Arbitrary gauge transformations, by any closed 2-form, make sense in the more
general context of Dirac structures (see e.g. \cite{BR03,SeWe}), where the notation $\pi_B=e^B\pi$ is justified (see \cite[Sec.~1]{Gua11}).

It was shown in \cite{BR03} that gauge equivalent Poisson structures
$\pi$ and $\pi_B$ are Morita equivalent with Morita bimodule given
by
\[
\xymatrix@R=10pt{
 &(\Sigma(M), \Omega-\s^*B) \ar[dl]_\t\ar[dr]^\s& \\
(M,\pi)& & (M,-\pi_B).}
\]
We will denote this Morita bimodule by $\Sigma(M)_B$.
\end{examp}

Two Morita equivalences $M\lmap S'\rmap \overline{N}$ and $N\lmap
S''\rmap \overline{P}$ can be composed:
\[
\xymatrix@R=10pt{
 &S'*S'' \ar[dl]\ar[dr]& \\
M& & \overline{P},}
\]
where the bimodule $S'*S''$ is defined as the quotient
\[
S'*S'':=\frac{S'\times_N \overline{S''}}{\Sigma(N)} .
\]
The symplectic form on $S'*S''$ is obtained by symplectic reduction.

An important feature of this operation is that it is associative
only up to natural equivalences of Morita bimodules:
\[ (S'*S'')*S'''\simeq S'*(S''*S''').\]
Note also that the symplectic groupoid acts as the unit for this
operation: for any Morita bimodule $M\lmap S\rmap \overline{N}$
there are natural isomorphisms:
\[ \Sigma(M)*S\simeq S,\quad S*\Sigma(N)\simeq S.\]
Moreover, given a Morita bimodule $M\lmap S\rmap \overline{N}$
the inverse Morita bimodule is $N\lmap \overline{S}\rmap
\overline{M}$, in the sense that we have natural isomorphisms
\[ S*\overline{S}\simeq\Sigma(M),\quad \overline{S}*S\simeq\Sigma(N).\]
For a Morita bimodule $S$, we may denote its inverse by $S^{-1}$.

\begin{examp}[Composition of Poisson diffeomorphisms and gauge transformations]
It should be clear that for Poisson diffeomorphisms
$\phi:(M,\pi_M)\to (N,\pi_N)$ and $\psi:(N,\pi_N)\to (P,\pi_P)$ we
have a natural isomorphism:
\[ \Sigma(M)_\phi * \Sigma(N)_\psi\simeq \Sigma(M)_{\psi\circ\phi} \]
Similarly, if $B_1,B_2\in\Omega^2(M)$ are closed 1-forms where $B_1$
determines a gauge equivalence of $\pi$ with $\pi_{B_1}$ and $B_2$
determines a gauge equivalence of $\pi_{B_1}$ with $\pi_{B_1+B_2}$,
then $B_1+B_2$ determines a gauge equivalence of $\pi$ with
$\pi_{B_1+B_2}$ and we have a natural equivalence
\[
\Sigma(M)_{B_1}*\Sigma(M)_{B_2}\simeq \Sigma(M)_{B_1+B_2}.
\]

More general compositions, involving both types of bimodules, will be discussed in Section \ref{sec:subgroups} below.
\end{examp}

\begin{rmk}\label{rmk:2grp}
The properties of composition of Morita bimodules described above reflect the fact that one may view (integrable) Poisson manifolds as objects in a category whose invertible morphisms are equivalence classes of Morita bimodules, see e.g. \cite[Sec.~2]{BW04}. More generally (see e.g. \cite{Land}) (integrable) Poisson manifolds may be seen as objects in a {\em bicategory} (a.k.a. a weak 2-category), with invertible 1-morphisms being Morita bimodules and 2-morphisms given by equivalences of bimodules.
\end{rmk}

\section{The Picard group}\label{sec:picard}

\subsection{Definition and first examples}\label{subsec:defpic}
The following definition was first proposed in \cite{BW04}:

\begin{defn}
The \textbf{Picard group} of a Poisson manifold $(M,\pi)$, denoted
by $\Pic(M)$, is its group of self Morita bimodules, modulo
isomorphisms of Morita bimodules.
\end{defn}

Note that Picard groups are the groups of automorphisms of Poisson manifolds regarded as objects in the category of Remark~\ref{rmk:2grp}; if one considers self Morita bimodules, rather than their isomorphism classes, one obtains a
(weak) 2-group (the 2-group of automorphisms of an object in the bicategory of Remark~\ref{rmk:2grp}). We will
not consider 2-categorical aspects of Morita equivalence in this paper, though it would be natural to extend our results in this direction.

We recall some examples of Poisson manifolds whose Picard groups are
known.

\begin{examp}[Symplectic manifolds]
\label{ex:sympl} Let $(M,\pi)$ be a Poisson manifold with a
non-degenerate Poisson structure. This means that $\pi^\sharp$ is an
isomorphism, so $\omega:=\pi^{-1}$ is a symplectic form. Let
$\phi\in\Aut(\pi_1(M))$ be an automorphism of the fundamental group
of $M$ and denote by $\widetilde{M}\to M$ the universal covering
space. The fundamental group $\pi_1(M)$ acts on
$\widetilde{M}\times\widetilde{M}$ by setting
\[ [\gamma]\cdot(m,n):=([\gamma]\cdot m,\phi([\gamma])\cdot n),\]
and we obtain the Morita bimodule
\[
\xymatrix{
 &\frac{\widetilde{M}\times\overline{\widetilde{M}}}{\pi_1(M)} \ar[dl]_{\text{pr}_1}\ar[dr]^{\text{pr}_2}& \\
M& & \overline{M}.}
\]
The trivial bimodule is obtained by taking $\phi$ to be the
identity. More generally, this bimodule is isomorphic to the trivial
bimodule if and only if $\phi$ is a inner automorphism of
$\pi_1(M)$.

It follows that the group of outer automorphisms $\OutAut(\pi_1(M))$
injects in $\Pic(M)$. One can show that, in fact, every self Morita
bimodule is isomorphic to one of this form, so that (see
\cite{BW04}):
\[
\Pic(M)\simeq\OutAut(\pi_1(M)).
\]
\end{examp}

\begin{examp}[Zero Poisson structure]
Let $M$ be any manifold with the zero Poisson structure $\pi=0$.
Recall that the symplectic groupoid of $M$ is
$\Sigma(M)=(T^*M,\Omega_\can)$, with $\s=\t=p$ the projection on the
base, and multiplication being addition on the fibers. One can
obtain self Morita bimodules of $M$ by subtracting from the
canonical symplectic form $\Omega_\can$ any 2-form $p^*B$, with
$B\in\Omega^2(M)$ a closed 2-form, and composing the source with any
diffeomorphism $\phi:M\to M$:
\[
\xymatrix{
 &(T^*M,\Omega_{\can}-p^*B) \ar[dl]_{p}\ar[dr]^{\phi\circ p}& \\
M& & \overline{M}.}
\]
Two such bimodules, induced by pairs $(\phi_1,B_1)$ and
$(\phi_2,B_2)$, are isomorphic if and only if $\phi_1=\phi_2$ and
$B_1-B_2$ is exact. Morover, the product of two such bimodules is
canonically isomorphic to the bimodule associated with the pair
$(\phi_1\circ\phi_2,B_1+\phi_1^*B_2)$.

We conclude that the semidirect product $\Diff(M)\ltimes H^2(M,\Rr)$
is a subgroup of $\Pic(M)$. One can show that, in fact, every self
Morita bimodule is isomorphic to one of this form, so that (see
\cite{BW04})
\[
\Pic(M)\simeq\Diff(M)\ltimes H^2(M,\Rr).
\]
\end{examp}

Picard groups have also been studied for a class of Poisson
structures on surfaces by Radko and Shlyakhtenko in \cite{RS}.

\subsection{Bisections}

In order to study Picard groups, it is convenient to generalize the
notion of bisection of groupoids: we define a \textbf{bisection of a
self Morita bimodule}
\[ M\stackrel{p}{\lmap} (S,\omega)\stackrel{q}{\rmap} \overline{M} \]
to be an embedded submanifold $L\subset S$ such that the
restrictions of both submersions $p$ and $q$ to $L$ induce
diffeomorphisms $L\to M$. In this case, we can choose an embedding
$b:M\to S$ such that $q\circ b=$id$_M$ and $p\circ b:M\to M$ is a
diffeomorphism. Conversely, the image of any such map is a
bisection, so we will identify bisections with maps $b:M\to S$
satisfying these two conditions. A \textbf{static bisection} is a
bisection $b:M\to S$ such that both $q\circ b=$id$_M$ and $p\circ
b=$id$_M$. A \textbf{Lagrangian bisection} is a bisection $b:M\to S$
such that $b^*\Omega=0$. Note that the trivial bimodule $\Sigma(M)$
always admits a static Lagrangian bisection -- namely, the identity
bisection.

\begin{prop}
\label{prop:group:bisections}
If $S_1$ and $S_2$ both admit bisections, then so does
$S_1*S_2^{-1}$. The elements of $\Pic(M)$ represented by
bimodules admitting a bisection form a normal subgroup. The same holds for Lagrangian (respectively, static)
bisections.
\end{prop}

\begin{proof}
If $b_1:M\to S_1$ and $b_2:M\to S_2$ are bisections of $S_1$ and
$S_2$, respectively, then the map
\[
 M\to S_1\times_{M} S_2,\ m\mapsto (b_1(\phi_2(m)),b_2(m)),
 \]
where $\phi_2=p_2\circ b_2$, induces a bisection
\[ b_1*b_2:M\to S_1*S_2=(S_1\times_{M} S_2)/\Sigma(M).\]
Clearly, a bisection of a bimodule $S$ is also a bisection of the
inverse $S^{-1}$. One can directly verify that if $b_1$ and $b_2$ are both Lagrangian (respectively, static), then $b_1*b_2$ is also Lagrangian (respectively, static).

Assume now that $S_0$ is a bimodule admitting a bisection $b_0:M\to S_0$. We claim that for any other
bimodule $M\stackrel{p}{\lmap} S\stackrel{q}{\rmap} \overline{M}$, the conjugate bimodule
\[ S* S_0* S^{-1}:= (S\times_M S_0\times_M S^{-1})/(\Sigma(M)\times\Sigma(M) \]
also admits a bisection. Here, the right hand side is the quotient associated with the free and proper right action defined, for $(u,y,u')\in S\times_M S_0\times_M S^{-1}$ and $(x_1,x_2)\in \Sigma(M)\times\Sigma(M)$, by
\[ (u,y,u')\cdot (x_1,x_2):=(u\cdot x_1,x_1^{-1}\cdot y \cdot x_2, x_2^{-1} \cdot u'). \]
It will be convenient to think of the identity map as an involutive automorphism:
\[ S\to S^{-1}, \quad u\mapsto \overline{u},\]
which switches the actions; hence for $u\in S$ and $x,y\in \Sigma(M)$ one has
\[
\overline{x\cdot u\cdot y}=y^{-1}\cdot \overline{u}\cdot x^{-1}.
\]
With this notation, we define
\begin{equation}
\label{eq:bisection}
b:M\to S*S_0*S^{-1}, \quad m\mapsto [b_0(m)\cdot \overline{u} \cdot b_0(n)^{-1}, b_0(n), u],
\end{equation}
where we chose any $u\in S^{-1}$ such that $p(u)=m$ and we set $n:=q(u)$.

We claim that \eqref{eq:bisection} is well-defined, i.e., it is independent of the choice of $u$. To see that, assume that we have $p(u)=m=p(u')$ and set $n:=p(u)$ and $n':=p(u')$. By principality, there exists $x\in\Sigma(M)$ such that $u'=x\cdot u$. Since $x$ is an arrow with target $n'$, by principality there is a unique $y\in\Sigma(M)$ such that
\[ b_0(n')=y\cdot b_0(n)\cdot x^{-1}. \]
But then
\begin{align*}
[b_0(m)\cdot \overline{u'} \cdot b_0(n')^{-1}, b_0(n'), u']&= [b_0(m)\cdot (\overline{u}\cdot x^{-1})\cdot x \cdot b_0(n)^{-1}y, y\cdot b_0(n)\cdot x^{-1}, x\cdot u]\\
&= [b_0(m)\cdot \overline{u} \cdot b_0(n)^{-1}y, y\cdot b_0(n)\cdot x^{-1}, x\cdot u]\\
&=[b_0(m)\cdot \overline{u} \cdot b_0(n), b_0(n), u],
\end{align*}
and this proves that \eqref{eq:bisection} is well-defined. Moreover, by choosing a local section of $p:S^{-1}\to M$ around $u$, we see that \eqref{eq:bisection} is smooth.

Finally, note that the conjugate bimodule
\[ M\stackrel{\hat{p}}{\lmap} S* S_0* S^{-1}\stackrel{\hat{q}}{\rmap} \overline{M} \]
has projections given by
\[ \hat{p}([u,y,v])=p(u),\quad \hat{q}([u,y,v])=p(v),\]
so that
\[ \hat{p}(b(m))=p(b_0(m)\cdot \overline{u} \cdot b_0(n))=p(b_0(m)),\quad \hat{q}(b(m))=p(u)=m,\]
and we conclude that \eqref{eq:bisection} defines a smooth bisection of $S* S_0* S^{-1}$. Moreover, if $b_0$ is a static bisection so is $b$.

 It is a direct verification that if $b_0:M\to S_0$ is a Lagrangian bisection then the bisection $b_0:M\to S* S_0* S^{-1}$ given by \eqref{eq:bisection} is also Lagrangian. So the elements of $\Pic(M)$ represented by bimodules admitting a Lagrangian bisection also form a normal subgroup.
\end{proof}


The problem of deciding if a bimodule admits a bisection is a rather
non-trivial topological problem. We will return to this issue in
Section~\ref{subsec:sympexample}, where examples will be discussed.

\subsection{Subgroups of the Picard group}
\label{sec:subgroups}

The Picard group of a Poisson manifold has some natural subgroups
with geometric meaning that we describe in this section.

\subsubsection*{Outer Poisson automorphisms}
If $\phi:M\to M$ is a Poisson automorphism of $(M,\pi)$ then the
self Morita bimodule $\Sigma(M)_\phi$ (see Example
\ref{ex:Poisson:diffeo}) represents an element of $\Pic(M)$. This
yields a group homomorphism
\[
\Aut(M)\to\Pic(M), \quad \phi\mapsto [\Sigma(M)_\phi].
\]
As observed in \cite{BW04}, the kernel of this homomorphism is
formed by the \textbf{inner Poisson automorphisms}, whose definition
we now recall.

A bisection of $\Sigma(M)\tto M$, viewed as an embedding
$b:M\to\Sigma(M)$ such that $\s\circ b$ is the identity on $M$ and
$\phi:=\t\circ b$ is a diffeomorphism of $M$, determines an inner
automorphism $\Phi_b:\Sigma(M)\to\Sigma(M)$,
\[
\Phi_b(x)=b(\t(x))\cdot x\cdot b(\s(x))^{-1},
\]
which covers the diffeomorphism $\phi$. When $b$ is a Lagrangian
bisection (i.e., $b^*\Omega=0$) the inner automorphism becomes a
symplectomorphism, while $\phi:M\to M$ becomes a Poisson
automorphism. The Poisson diffeomorphisms that are obtained in this
way form the subgroup of inner Poisson automorphisms, denoted by
$\InnAut(M)$.

We conclude that the group of \textbf{outer Poisson automorphisms}
is a subgroup of the Picard group:
\begin{equation}
\label{eq:outer:Poisson}
\OutAut(M):=\frac{\Aut(M)}{\InnAut(M)}=\frac{\Aut(M)}{\LBis(\Sigma(M))}\subset
\Pic(M).
\end{equation}

\subsubsection*{Outer gauge transformations}
Let $B\in\Omega^2(M)$ be a 2-form defining a self gauge equivalence
of $(M,\pi)$, so that $\pi_B=\pi$. This happens precisely when $B$
is a closed, basic 2-form, i.e., when
\[
\d B=0,\quad i_{X_f} B=0, \qquad\forall f\in C^\infty(M).
\]
Such a 2-form defines the self Morita bimodule $\Sigma(M)_B$ (see
Example \ref{ex:gauge:transf}) and hence determines an element of
$\Pic(M)$. This yields a group homomorphism
\[ \Omega_{\cl,\bas}^2(M)\to\Pic(M), \quad B\mapsto \Sigma(M)_B.\]
We will see later in Proposition \ref{prop:trivial:module} that a
form $B\in \Omega_{\cl,\bas}^2(M)$ is in the kernel of this map if
and only if there exists a static bisection $b:M\to\Sigma(M)$ such
that $b^*\Omega=B$. We denote the group of static bisections  by
$\IBis(\Sigma(M))$, and we have a natural map
$\IBis(\Sigma(M))\to \Omega_{\cl,\bas}^2(M)$ given by $b\mapsto b^*\Omega$.

We conclude that the group of \textbf{outer gauge transformations}
is a subgroup of the Picard group:
\begin{equation}
\label{eq:outer:gauge}
\OutGaug(M):=\frac{\Omega_{\cl,\bas}^2(M)}{\IBis(\Sigma(M))}\subset
\Pic(M).
\end{equation}
Notice that if $\alpha\in\Omega_{\bas}^1(M)$ then
$\exp(\al)\in\IBis(\Sigma(M))$. By Proposition
\ref{prop:bisections:form}, the map $\IBis(\Sigma(M))\to
\Omega_{\cl,\bas}^2(M)$ maps $\exp(\al)$ to $\d\alpha$. Hence,
$\d\Omega_{\bas}^1(M)$ is contained in the kernel of
$\Omega_{\cl,\bas}^2(M)\to\Pic(M)$, and we conclude that there is a
group homomorphism from the \textbf{second basic cohomology group},
viewed as an abelian group, to the Picard group:
\[
H^2_{\bas}(M)\to \OutGaug(M)\hookrightarrow \Pic(M).
\]

\subsubsection*{Gauge equivalence up to Poisson diffeomorphism}
The subgroups of $\Pic(M)$ that we considered above can be combined
into a larger subgroup, which will play a key role in our study of
the Picard group.

More precisely, we start with the following data:
\begin{enumerate}[(a)]
\item A closed 2-form $B\in\Omega^2(M)$ such that $I+B^\flat\circ\pi^\sharp$ is invertible;
\item A diffeomorphism $\phi$ such that $\phi_*\pi_B=\pi$.
\end{enumerate}
Then the composition of the corresponding Morita bimodules (see
Examples \ref{ex:Poisson:diffeo} and \ref{ex:gauge:transf}),
\[
\xymatrix@C=0 pt{
 &(\Sigma(M), \Omega-\s^*B) \ar[dl]_\t\ar[dr]^\s  \\
(M,\pi)& & (M,-\pi_{B}), }\;
\xymatrix@C=0 pt{
&(\Sigma(M), \Omega+\t^*B-\s^*B) \ar[dl]_{\t}\ar[dr]^{\phi\circ\s} \\
(M,\pi_B) & & (M,-\pi), }
\]
yields a self Morita bimodule of $(M,\pi)$ which one directly checks
to be canonically isomorphic to the bimodule
\begin{equation}\label{eq:phiB}
\xymatrix{
 &(\Sigma(M), \Omega-\s^*B) \ar[dl]_{\t}\ar[dr]^{\phi\circ\s}& \\
(M,\pi)& & (M,-\pi).}
\end{equation}
We will denote this self Morita bimodule by $\Sigma(M)_{(\phi,B)}$.

If we are given two pairs $(\phi_1,B_1)$ and $(\phi_2,B_2)$
satisfying (a) and (b) above, then one verifies that the product
bimodule $\Sigma(M)_{(\phi_1,B_1)}*\Sigma(M)_{(\phi_2,B_2)}$ is
canonically isomorphic to the bimodule
$\Sigma(M)_{(\phi_1\circ\phi_2,B_1+\phi_1^*B_2)}$. Hence, we
introduce the subgroup $\G_{\pi}(M)\subset
\Diff(M)\ltimes\Omega_{\cl}^2(M)$ given by
\[
\G_{\pi}(M):=\{(\phi,B)~|~(I+B^\flat\circ\pi^\sharp)\text{ is
invertible and }\phi_*\pi_B=\pi\}.
\]
and we have the group homomorphism
\begin{equation}
\label{eq:gauge:picard}
\G_{\pi}(M)\to \Pic(M),\ (\phi,B)\mapsto [\Sigma(M)_{(\phi,B)}].
\end{equation}

In the next section we will determine the kernel and the image of
this homomorphism. In particular, we will see that the image is
important for the understanding of $\Pic(M)$.

Note that the composition of the inclusions
\begin{align*}
\Aut(M)&\hookrightarrow \G_{\pi}(M),\ \phi\mapsto (\phi,0),\\
\Omega^2_{\bas,\cl}(M)&\hookrightarrow \G_{\pi}(M),\ B\mapsto (I,B),
\end{align*}
with the homomorphism $\G_{\pi}(M)\to \Pic(M)$ give rise to the two
subgroups $\OutAut(M)$ and $\OutGaug(M)$ of $\Pic(M)$ that we saw
above (see \eqref{eq:outer:Poisson} and \eqref{eq:outer:gauge}).

\subsection{The Picard group exact sequence}

For the trivial bimodule $\Sigma(M)$ the bisections
form a group $\Bis(\Sigma(M))$. Recall that $\LBis(\Sigma(M))$
denotes the subgroup of Lagrangian bisections, while
$\IBis(\Sigma(M))$ is the subgroup of static bisections; we will
also consider the intersection of these subgroups, denoted by
$\IsoL(\Sigma(M))$.

Note that we have a homomorphism of groups
\begin{equation}
\label{eq:bis:gauge} \Bis(\Sigma(M))\to \G_\pi(M),\ b\mapsto
(\t\circ b, b^*\Omega),
\end{equation}
whose kernel is $\IsoL(\Sigma(M))$ (see also \cite[Sec.~1.6]{AM} for another context in which
this homomorphism arises).
One can put together the group
homomorphisms \eqref{eq:gauge:picard} and \eqref{eq:bis:gauge} into
an exact sequence, which is the main tool to compute Picard groups:

\begin{thm}
\label{thm:main} There is an exact sequence of groups
\begin{equation}
\label{eq:main:exact:seq}
\xymatrix{
1\ar[r] &\IsoL(\Sigma(M))\ar[r]& \Bis(\Sigma(M))\ar[r]&\G_\pi(M)\ar[r]&\Pic(M)}
\end{equation}
such that the image of the last map is the normal subgroup formed by self
Morita bimodules which admit a bisection. In this sequence, the
first arrow is the inclusion, while the second and third arrows are
the maps \eqref{eq:bis:gauge} and \eqref{eq:gauge:picard},
respectively.
\end{thm}

The exact sequence \eqref{eq:main:exact:seq} has two interesting
exact subsequences corresponding, respectively, to the cases $B=0$
and $\phi=$id in \eqref{eq:bis:gauge}. The following sequence already
appears in \cite{BW04}:

\begin{cor}
The exact sequence \eqref{eq:main:exact:seq} has an exact
subsequence
\begin{equation}
\label{eq:Diff:exact:seq} \xymatrix@C24pt{ 1\ar[r]
&\IsoL(\Sigma(M))\ar[r]&
\LBis(\Sigma(M))\ar[r]&\Aut(M)\ar[r]&\Pic(M),}
\end{equation}
so that that the image of the last map is the normal subgroup of $\Pic(M)$
formed by the self Morita bimodules which admit a Lagrangian
bisection, and it coincides with the group of outer automorphisms
$\OutAut(M)$ given by \eqref{eq:outer:Poisson}.
\end{cor}

\begin{cor}
The exact sequence \eqref{eq:main:exact:seq} has an exact
subsequence
\begin{equation}
\label{eq:Gauge:exact:seq} \xymatrix@C20pt{ 1\ar[r]
&\IsoL(\Sigma(M))\ar[r]&
\IBis(\Sigma(M))\ar[r]&\Omega^2_{\cl,\bas}(M)\ar[r]&\Pic(M),}
\end{equation}
so that the image of the last map is the normal subgroup of $\Pic(M)$
formed by the self Morita bimodules which admit a static bisection,
and it coincides with the group of outer gauge transformations
$\OutGaug(M)$ given by \eqref{eq:outer:gauge}.
\end{cor}

\begin{proof}[Proof of Theorem \ref{thm:main}] By Proposition \ref{prop:group:bisections} all that remains to be shown is:
\begin{enumerate}[(a)]
\item the sequence \eqref{eq:main:exact:seq} is exact at $\G_\pi(M)$, and
\item the image of the last map is the subgroup defined by self Morita bimodules which admit a bisection.
\end{enumerate}

The proofs of these two statements are given in the next two
propositions.

\begin{prop}
\label{prop:trivial:module} The self Morita bimodule
$\Sigma(M)_{(\phi,B)}$ is isomorphic to the trivial bimodule if and
only if there exists a bisection $b\in\Bis(\Sigma(M))$ such that
\begin{equation}
\label{eq:trivial:gauge} \phi=\t\circ b\text{ \, and \,
}b^*\Omega=B.
\end{equation}
\end{prop}

\begin{proof}
Let $\Psi:\Sigma(M)_{(\phi,B)}\to \Sigma(M)$ be an isomorphism to
the trivial bimodule:
\[
\xymatrix{
(\Sigma(M),\Omega-\s^*B)\ar@<3pt>[d]^{\phi\circ\s}\ar@<-3pt>[d]_{\t} \ar[rr]^{\Psi}&& (\Sigma(M),\Omega)\ar@<3pt>[d]^{\s}\ar@<-3pt>[d]_{\t}\\
M\ar@{=}[rr]&& M
}
\]
(by definition, $\Psi$ covers the identity). Recall that the
identity bisection $\eps:M\to\Sigma(M)$ is Lagrangian. Then if we
define $b:M\to\Sigma(M)$ by $b:=\Psi^{-1}\circ \eps\circ\phi$, we
check immediately that
\[
\s\circ b=\text{id}_M,\quad \t\circ b=\phi\text{\; and \,
}b^*\Omega=B,
\]
which shows that $b\in\Bis(\Sigma(M))$ is a bisection for which
\eqref{eq:trivial:gauge} holds.

Conversely, given a bimodule $\Sigma(M)_{(\phi,B)}$ for which there
exists $b\in\Bis(\Sigma(M))$ such that \eqref{eq:trivial:gauge} is
satisfied, we obtain an isomorphism $\Psi:\Sigma(M)_{(\phi,B)}\to
\Sigma(M)$ by taking right translation by $b^{-1}$:
\[ \Psi(x):=x\cdot b(\s(x))^{-1}. \]
In fact, one directly checks that $\t\circ\Psi=\t$ and
$\s\circ\Psi=\phi\circ\s$. If we define
\[ \Delta:\Sigma(M)\to\Sigma(M)\times\Sigma(M),\quad x\mapsto (x,\iota(b(\s(x)))),\]
then we can write $\Psi=m\circ\Delta$ (here $m$ is the groupoid
multiplication and $\iota$ is the inversion). The multiplicative
property \eqref{eq:multiplicative} of the symplectic form $\Omega$
and the fact that the inversion map is anti-symplectic imply that
\begin{align*}
\Psi^*\Omega&=\Delta^*m^*\Omega\\
            &=\Delta^*(\pr_1^*\Omega+\pr_2^*\Omega)\\
            &=\Omega+\s^*b^*\iota^*\Omega\\
            &=\Omega-\s^*b^*\Omega=\Omega-\s^*B.
\end{align*}
So $\Psi$ is indeed a symplectomorphism.
\end{proof}

Clearly, self Morita bimodules of the form $\Sigma(M)_{(\phi,B)}$
admit bisections. To complete the proof of Theorem \ref{thm:main},
it remains to show the following:

\begin{prop}
\label{prop:canonical:form}
Let
\[ M\stackrel{p}{\lmap} (S,\omega)\stackrel{q}{\rmap} \overline{M} \]
be a self Morita bimodule which admits a bisection $b:M\to S$. If we
set $\phi:=(p\circ b)^{-1}$ and $B:=-\phi^*b^*\omega$, then there
exists a unique isomorphism of Morita bimodules $\Psi:S\to
\Sigma(M)_{(\phi,B)}$.
\end{prop}

For the proof of this proposition, we need the following lemma.
Consider the closed 2-form
$\widetilde{\omega}:=\omega-q^*b^*\omega$.

\begin{lem}
\label{lem:symp} $\widetilde{\omega}$ is a symplectic form and
$p:(S,\widetilde{\omega})\to (M,\pi)$ is a complete symplectic
realization.
\end{lem}

\begin{proof}
To check that $\widetilde{\omega}$ is symplectic, we observe that it
suffices to verify that  $\widetilde{\omega}_{b(m)}$ is
non-degenerate for all $m\in M$. Indeed, if this is the case, then
the (pointwise) pushforward of the bivector
$(\widetilde{\omega}_{b(m)})^{-1}$ agrees with $(\pi_B)_m$, the
gauge transformation of $\pi$ by the closed 2-form $B=-b^*\omega$ at
$m$, see \cite[Lemma~2.12]{BR03}; in particular, $B$ defines a gauge
equivalence of $\pi$ (as in Example~\ref{ex:gauge:transf}), and this
guarantees that $\widetilde{\omega}$ is nondegenerate everywhere,
again as a consequence of \cite[Lemma~2.12]{BR03}.

To prove the nondegeneracy of $\widetilde{\omega}_{b(m)}$, we first
claim that $\Ker\widetilde{\omega}_{b(m)}\subset\Ker\d_{b(m)}p$. In
fact, if $v\in\Ker\widetilde{\omega}_{b(m)}$ we have for any
$w\in\Ker\d_{b(m)}q$:
\begin{align*}
\omega_{b(m)}(v,w)&=\widetilde{\omega}_{b(m)}(v,w)+(q^*b^*\omega)_{b(m)}(v,w)\\
                                  &=0+\omega_{b(m)}(\d_m b\cdot\d_{b(m)}q\cdot v,\d_m b\cdot\d_{b(m)}q\cdot w)=0.
\end{align*}
This means that $v$ belongs to
$(\Ker\d_{b(m)}q)^{\perp_\omega}=\Ker\d_{b(m)}p$, as claimed. Next,
we observe that $b$ is an isotropic section for
$\widetilde{\omega}$:
\begin{align*}
b^*\widetilde{\omega}&=b^*\omega-b^*q^*b^*\omega\\
                                         &=b^*\omega-(b\circ q\circ b)^*\omega\\
                                         &=b^*\omega-b^*\omega=0,
\end{align*}
where we use that $q\circ b=$id$_M$. Observing that the dimension of
the section $b$ is $\dim M=\frac{1}{2}\dim S$, and that the section
$b$ is transverse to the $p$-fibers, we conclude that $b$ is a
Lagrangian section for $\widetilde{\omega}$ and that
$\Ker\widetilde{\omega}_{b(m)}=0$, for all $m\in M$.

To finish the proof of the lemma, we have to show that
$p:(S,\widetilde{\omega})\to (M,\pi)$ is a complete Poisson map.
This follows from the fact that for any $f:M\to\Rr$ the hamiltonian
vector fields for $f\circ p$ relative to $\omega$ and relative to
$\widetilde{\omega}$ coincide and the fact that $p:(S,\omega)\to
(M,\pi)$ is already a complete Poisson map.
\end{proof}

\begin{proof}[Proof of Proposition~\ref{prop:canonical:form}]
Note that $b\circ\phi:M\to S$ is a Lagrangian section of the complete symplectic realization
$p:(S,\widetilde{\omega})\to (M,\pi)$. By Theorem
\ref{thm:canonical:realz}, we have an isomorphism of symplectic
realizations,
\[
\xymatrix{
(S,\tilde{\omega})\ar[d]_{p} \ar[rr]^{\Psi}&& (\Sigma(M),\Omega)\ar[d]_{\t}\\
M\ar@{=}[rr]&& M, }
\]
which maps the Lagrangian bisection $b\circ\phi:M\to S$ to the
identity bisection $\eps:M\to\Sigma(M)$.

By the very definition of a self Morita bimodule, the $q$-fibers are
$\omega$-symplectic orthogonal to the $p$-fibers, and it follows
that they are also $\widetilde{\omega}$-symplectic orthogonal. On
the other hand, the source and target fibers are $\Omega$-symplectic
orthogonal. Since $\Psi:(S,\widetilde{\omega})\to
(\Sigma(M),\Omega)$ is a symplectomorphism, we conclude that there
must exist a diffeomorphism $\psi:M\to M$ such that $q=\psi\circ\s\circ\Psi$.
However, since $\Psi$ maps the bisection $b\circ\phi:M\to S$ to the
bisection $\eps:M\to\Sigma(M)$, we find that:
\begin{align*}
\phi&=q\circ b\circ\phi\\
    &=\psi\circ\s\circ\Psi\circ b\circ\phi\\
    &=\psi\circ\s\circ\eps=\psi.
\end{align*}
Finally, we observe that:
\begin{align*}
\Psi^*(\Omega-\s^*B)&=\Psi^*\Omega-(\s\circ\Psi)^*B\\
                                         &=\widetilde{\omega}+(\phi\circ\s\circ\Psi)^*b^*\omega\\
                                         &=\widetilde{\omega}+q^*b^*\omega=\omega.
\end{align*}
This shows that we have an isomorphism of self-Morita bimodules:
\[
\xymatrix{
(S,\omega)\ar@<3pt>[d]^{q}\ar@<-3pt>[d]_{p} \ar[rr]^{\Psi}&& (\Sigma(M),\Omega-\s^*B)\ar@<3pt>[d]^{\phi\circ\s}\ar@<-3pt>[d]_{\t}\\
M\ar@{=}[rr]&& M
}
\]
as claimed.
\end{proof}

This concludes the proof of Theorem \ref{thm:main}.
\end{proof}

\subsection{$\Pic(M)$ as a topological group}
\label{sec:top:group}

In order to introduce a natural topology on the Picard group,
we start by introducing a topology on the space of self Morita
bimodules.

We will use the Whitney $C^\infty$-topology on the space $C^\infty(M,N)$ of smooth maps between two smooth manifolds (see
\cite{GG86,Hir94}). We recall that for a (possibly non-compact) manifold $M$ a sequence $\{\phi_n\}\subset C^\infty(M,N)$ converges to $\phi$
in the Whitney $C^k$-topology if there is a compact set $K\subset M$ such that the $k$-th jets $j^k\phi_n$ converge uniformly to  $j^k\phi$ on $K$ and all but a finite number of $\phi_n$'s equal $\phi$ outside $K$.

A neighborhood of a Morita bimodule $M\stackrel{p}{\lmap}
(S,\omega)\stackrel{q}{\rmap} \overline{M}$ consists of all bimodules
$M\stackrel{p'}{\lmap} (S',\omega')\stackrel{q}{\rmap}
\overline{M}$, where:
\begin{itemize}
\item The manifolds $S$ and $S'$ coincide;
\item The symplectic forms $\omega'$ belong to a neighborhood of $\omega$ in the
Whitney $C^\infty$-topology;
\item The submersions $(p',q')$ belong to a neighborhood of $(p,q):S\to M\times M$ in the
Whitney $C^\infty$-topology.
\end{itemize}

We endow the Picard group with the quotient topology induced from self-Morita bimodules.

Recall from Theorem \ref{thm:main} that the image of the group
homomorphism
\[
\G_\pi(M)\to \Pic(M),\ (B,\phi)\mapsto [\Sigma(M)_{(B,\phi)}]
\]
is the subgroup of $\Pic(M)$ given by the self Morita bimodules
which admit a bisection. On $\G_\pi(M)$ we consider the subspace
topology induced from the Whitney $C^\infty$-topology on the space
$\Omega^2(M)\times\Diff(M)$, with respect to which $\G_\pi(M)$ is a
topological group (see e.g. \cite{Mil83}).

Let $\Pic(M)^0$ denote the connected component of the identity of $\Pic(M)$.

\begin{prop}
\label{prop:topology:bisections} The map $\G_\pi(M)\to \Pic(M)$ is
continuous and open. Its image, the subgroup of $\Pic(M)$ formed by
the self Morita bimodules which admit a bisection, is an open subset
and a topological group containing $\Pic(M)^0$.
\end{prop}

\begin{proof}
The set of self Morita bimodules which admit a bisection is open
because the set of diffeomorphisms is an open subset in the space of
all smooth maps, in the Whitney $C^\infty$-topology (see, e.g.,
\cite[Theorem II.1.7]{Hir94}). Hence, the corresponding subgroup in
$\Pic(M)$ is an open subset.  Moreover, up to diffeomorphism,
Proposition \ref{prop:canonical:form} shows that every such bimodule
is isomorphic to one of the form $\Sigma(M)_{(B,\phi)}$. Hence, the
definition of the topology shows that the elements
$[\Sigma(M)_{(B,\phi)}]\in\Pic(M)$ form an open set whenever
$(B,\pi)$ range over an open set in $\Omega^2(M)\times\Diff(M)$.  We
conclude that the map $\G_\pi(M)\to \Pic(M)$ is both continuous and
open.
\end{proof}

We conjecture that $\Pic(M)$ itself is a topological group. We can
prove this assertion when $\Sigma(M)$ is a compact groupoid, in
which case we can rely on recent rigidity results from \cite{CrStMe15,dHF15}. We start
with a preliminary observation.

\begin{lem}\label{lem:compact}
For a Poisson manifold $(M,\pi)$ the following statements are
equivalent:
\begin{enumerate}[(i)]
\item $\Sigma(M)$ is compact;
\item every self Morita bimodule is compact;
\item there exists one self Morita bimodule which is compact.
\end{enumerate}
\end{lem}

\begin{proof}
For a Morita bimodule $M\stackrel{p}{\lmap} S\stackrel{q}{\rmap}
\overline{M}$ the fibers of the surjective submersions $p$ and $q$
are diffeomorphic to the source-fibers of $\Sigma(M)$.
\end{proof}


\begin{thm}
\label{thm:Picard:topological:group}
If $\Sigma(M)$ is compact, then $\Pic(M)$ is a topological group.
\end{thm}

\begin{proof}
Proposition \ref{prop:topology:bisections} already shows that the
subgroup of $\Pic(M)$ formed by the self Morita bimodules which
admit a bisection is a topological group which contains $\Pic(M)^0$.
In order to conclude that $\Pic(M)$ itself is a topological group,
by \cite[Chap. III.1.2]{bourbaki} it suffices to check the following property:

\begin{lem}
For a fixed $[S]\in \Pic(M)$, the left and right translations by $[S]$:
\[ \Pic(M)\to\Pic(M),\quad [S']\mapsto [S*S'], \quad [S']\mapsto [S'*S], \]
are homeomorphisms.
\end{lem}

\begin{proof}
We consider the case of left translations by $S$. Right translations
are treated exactly the same way, with the role of the maps
exchanged.

Fix a self Morita bimodule $S$. The result will follow from the fact
that any self Morita bimodule $S'$ sufficiently close to $S$ is
isomorphic to a self-Morita bimodule of the form
$S*\Sigma(M)_{(B,\phi)}$, and that a sufficiently small neighborhood
of the identity is spanned by self Morita bimodules of the form
$\Sigma(M)_{(B,\phi)}$.

Our compactness assumption now allows us to resort to the fact that
compact groupoids are rigid (see \cite{CrStMe15,dHF15}). Hence, if the bimodule
$S'$ is close enough to $S$ then there is an isomorphism of the
action groupoids (we choose the left actions),
\[
\xymatrix{
\Sigma(M)\ltimes S'\ar@<3pt>[d]\ar@<-3pt>[d] \ar[rr]^{\hat{\Phi}}&& \Sigma(M)\ltimes S\ar@<3pt>[d]\ar@<-3pt>[d]\\
S'\ar[rr]_{\Phi}&& S, }
\]
which induces an isomorphism on the orbit spaces of these groupoids
which we denote by $\phi:M\to M$ (the compactness of the action
groupoids follows from Lemma~\ref{lem:compact}).
 It follows that under the isomorphism $\Phi:S'\to S$ the (left)
action of $\Sigma(M)$ on $S'$ is taken to a left action of
$\Sigma(M)$ on $S$ with structure maps
\[
\xymatrix{
S'\ar@<3pt>[d]^{q'}\ar@<-3pt>[d]_{p'} \ar[rr]^{\Phi}&& S\ar@<3pt>[d]^{\phi^{-1}\circ q}\ar@<-3pt>[d]_{p}\\
M\ar@{=}[rr]&& M
}
\]
Moreover, this action is symplectic for the symplectic form
$(\Phi^{-1})^*\omega'$. Now observe that, since infinitesimal
generators  of the action coincide, for any $\al\in\Omega^1(M)$ we
have:
\[
i_{X_\al}\omega=p^*\al,\quad i_{X_\al}(\Phi^{-1})^*\omega'=p^*\al.
\]
It follows that $i_{X_\al}(\omega-(\Phi^{-1})^*\omega')=0$. Since
the vector fields $X_\al$ span the fibers of $q$, we conclude that
the closed form $\omega-(\Phi^{-1})^*\omega'$ is $q$-basic, and
there is a closed form $B\in\Omega^2(M)$ such that
\[
\omega-(\Phi^{-1})^*\omega'=q^*B.
\]
This proves the claim that any self Morita bimodule $S'$
sufficiently close to $S$ is isomorphic to a self-Morita bimodule of
the form $S*\Sigma(M)_{(B,\phi)}$ and completes the proof of the lemma.
\end{proof}

This concludes the proof of Theorem \ref{thm:Picard:topological:group}.
\end{proof}

When $\Sigma(M)$ is compact, we can refine the statement in
Theorem~\ref{thm:main}:

\begin{cor}
There is an exact sequence of topological groups:
\[
\xymatrix{
1\ar[r] &\IsoL(\Sigma(M))\ar[r]& \Bis(\Sigma(M))\ar[r]&\G_\pi(M)\ar[r]&\Pic(M)}
\]
such that the image of the last arrow is the open normal subgroup formed by self Morita bimodules which admit a bisection.
\end{cor}

\begin{proof}
The groups $\IsoL(\Sigma(M))$, $\Bis(\Sigma(M))$ and $\G_\pi(M)$ are
spaces of maps and have evident $C^\infty$-topologies, which make
them topological groups and for which the various maps in the
sequence are continuous homomorphisms.
\end{proof}

Similarly, it should be clear that the sequences \eqref{eq:Diff:exact:seq}
\eqref{eq:Gauge:exact:seq} are also exact sequences of topological
groups.

\section{The Picard Lie algebra}

In the previous section, we saw that $\Pic(M)$ carries a natural
topology. It is in fact useful to think of it as a (possibly
infinite-dimensional) Lie group. In this section we identify its Lie
algebra, $\pic(M)$, which is easier to compute and provides
information about the "size" of $\Pic(M)$.

Our approach to describe $\pic(M)$ is as follows. We know from
Proposition \ref{prop:topology:bisections} that $\Pic(M)^0$, the connected
component of the identity of $\Pic(M)$, lies in the image of the
homomorphism \eqref{eq:gauge:picard}:
$$
\G_\pi(M)\to \Pic(M).
$$
So we identify $\pic(M)$ in two
steps: first, we describe the Lie algebra $\frakg_\pi(M)$ of
$\G_\pi$; second, we identify the ideal $\mathfrak{I}\subset \frakg_\pi(M)$ corresponding to the Lie algebra of the kernel of the homomorphism
above. We  then set $\pic(M) := \frakg_\pi(M)/\mathfrak{I}$. As a
consequence, we will see that $\pic(M)$ fits into an exact sequence
of Lie algebras that is the infinitesimal version of the exact
sequence \eqref{eq:main:exact:seq}.

In order to fully justify this viewpoint, one needs to regard the
topological groups in the exact sequence \eqref{eq:main:exact:seq}
as (infinite-dimensional) Lie groups. When $\Sigma(M)$ is compact,
so that $M$ is also compact, this is indeed possible and more or
less standard (see, e.g., \cite{Mil83}). When $\Sigma(M)$ is not
compact, one needs some more sophisticated machinery, such as the
convenient setting described in \cite{KM97}. In what follows we will
proceed formally, e.g., describing the Lie algebra $\frakg_\pi(M)$ by
differentiating smooth paths in $\G_\pi(M)$.

\subsection{The gauge Lie algebra $\frakg_\pi(M)$}
\label{sec:gauge:Lie:algebra}

Recall that the group of  gauge transformations of $\pi$ up to
Poisson diffeomorphisms, $\G_\pi(M)\subset
\Diff(M)\ltimes\Omega^2(M)$, consists of pairs $(\phi,B)$, where:
\begin{enumerate}[(a)]
\item $B$ is a closed 2-form on $M$ such that $I+B^\flat\circ\pi^\sharp:T^*M\to T^*M$ is
invertible;
\item $\phi:M\to M$ is a diffeomorphism such that $\phi_*\pi_B=\pi$.
\end{enumerate}

In order to find its Lie algebra $\frakg_\pi(M)$, we consider a
smoothly parameterized curve $(\phi_t,B_t)\in \G_\pi(M)$ starting at
the identity $(I,0)$. Differentiating at $t=0$, we obtain a pair
$(Z,\beta)\in\X(M)\times \Omega^2(M)$, which must satisfy the
infinitesimal version of conditions (a) and (b):
\[
\frakg_\pi(M)=\{(Z,\be)\in\X(M)\times \Omega^2(M):\d\beta=0,\
\Lie_Z\pi=\pi^\sharp(\beta)\}.
\]
Also, if $(Z_1,\beta_1)$ and $(Z_2,\beta_2)$ are two elements of
$\frakg_\pi(M)$, then their Lie bracket is the semidirect product
Lie bracket:
\[
[(Z_1,\beta_1),(Z_2,\beta_2)]_{\frakg_\pi(M)}=([Z_1,Z_2],\Lie_{Z_1}\beta_2-\Lie_{Z_2}\beta_1).
\]

The Lie algebra $\frakg_\pi(M)$ contains two important Lie
subalgebras:
\begin{itemize}
\item the \textbf{Lie algebra of the group $\Aut(M)$ of Poisson
diffeomorphisms}, which is given by the subspace of Poisson vector
fields:
\[ \X_\pi(M)=\{Z\in\X(M)~:~\Lie_Z\pi=0\}.\]
The inclusion $ \X_\pi(M)\hookrightarrow \frakg_\pi(M)$ is given by $Z\mapsto (Z,0)$.

\item the \textbf{Lie algebra of the group of self gauge transformations}.
Since this group is abelian and 1-connected, its Lie algebra is
given by the 2-forms $\beta\in\Omega^2(M)$ such that
\[ \d\beta=0\text{ and }\pi^\sharp(\beta)=0,
\]
which is easily seen to coincide with the Lie algebra of closed
basic 2-forms:
\[
\Omega_{\cl,\bas}^2(M)=\{\beta\in\Omega^2(M)~:~\d\beta=0,~ i_{X_f}\beta=0,\forall f\in C^\infty(M)\}.
\]
The inclusion $\Omega_{\cl,\bas}^2(M)\hookrightarrow \frakg_\pi(M)$ is given by $\beta\mapsto (0,\beta)$.
\end{itemize}

\subsection{The Picard Lie algebra}

In order to define the Picard Lie algebra as a quotient of
$\frakg_\pi(M)$, let us consider the kernel of the group
homomorphism $\G_\pi(M)\to \Pic(M)$. From the exact sequence
\eqref{eq:main:exact:seq}, we know that this kernel agrees with the
image of the homomorphism $\Bis(\Sigma(M))\to \G_\pi(M)$ defined in
\eqref{eq:bis:gauge}. The infinitesimal counterpart of this group
homomorphism is a Lie algebra homomorphism:
\[
\Omega^1(M)\to \frakg_\pi(M),
\]
where we have used the well-known fact that, regarding
$\Bis(\Sigma(M))$ as an (infinite-dimensional) Lie group, its Lie
algebra is given by the space of sections of the Lie algebroid of
$\Sigma(M)$ (which is $\Omega^1(M)$, with the Koszul bracket
\eqref{eq:Koszul:bracket}, in the case of a Poisson manifold).

To describe this infinitesimal map explicitly, let $b_t$ be a family
of bisections defined by the flow of a 1-form $\eta\in\Omega^1(M)$ (viewed as a section of the Lie algebroid of $\Sigma(M)$), and let
\[
\phi_t:=\t\circ b_t,\quad B_t:=b_t^*\Omega
\]
be the corresponding path in $\G_\pi(M)$, with velocity $(Z,\beta)$
at $t=0$. One can directly verify that the relation between $\eta$
and $(Z,\beta)$ is given by
\[
Z=\pi^\sharp(\eta),\quad \beta=\d\eta,
\]
which defines the desired map $\Omega^1(M)\to \frakg_\pi(M)$. Note
that the pair $(\pi^\sharp(\eta),\d\eta)$ is indeed in
$\frakg_\pi(M)$, since for any 1-form $\eta\in\Omega^1(M)$, the
vector field $\pi^\sharp(\eta)$ satisfies:
\[
\Lie_{\pi^\sharp(\eta)}\pi=\pi^\sharp(\d\eta).
\]

Hence the image of $\Omega^1(M)$ in $\frakg_\pi(M)$, denoted by
$\mathfrak{I}$, is given by
$$
\mathfrak{I}:=\{ (\pi^\sharp(\eta),\d\eta) \,:\, \eta\in
\Omega^1(M)\} \subset \mathfrak{g}_\pi(M).
$$

\begin{lem}\label{lem:ideal}
The subspace $\mathfrak{I}\subset \mathfrak{g}_\pi(M)$ is an ideal.
\end{lem}

\begin{proof}
For $(Z,\beta)\in \mathfrak{g}_\pi(M)$, we have
$$
[(Z,\beta),(\pi^\sharp(\eta),\d \eta)]_{\mathfrak{g}_\pi(M)} =
([Z,\pi^\sharp(\eta)],\Lie_{Z}\d \eta -
\Lie_{\pi^\sharp(\eta)}\beta).
$$
Since $\d \beta=0$, we can write $\Lie_{\pi^\sharp(\eta)}\beta = \d
i_{\pi^\sharp(\eta)}\beta$, so
$$
\Lie_{Z}\d \eta - \Lie_{\pi^\sharp(\eta)}\beta = \d(\Lie_{Z} \eta -
i_{\pi^\sharp(\eta)} \beta).
$$
On the other hand:
\begin{align*}
[Z,\pi^\sharp(\eta)] &= \Lie_Z i_\pi \eta = i_\pi \Lie_Z\eta +
i_{[Z,\pi]}\eta \\
& = i_\pi \Lie_Z\eta + i_{\pi^\sharp(\beta)}\eta =
\pi^\sharp(\Lie_Z\eta - i_{\pi^\sharp(\eta)}\beta),
\end{align*}
since we have $[Z,\pi] = \Lie_Z \pi =  \pi^\sharp(\beta)$. Hence, $[(Z,\beta),(\pi^\sharp(\eta),\d \eta)]_{\mathfrak{g}_\pi(M)}\in \mathfrak{I}$.
\end{proof}

We are now ready to define the infinitesimal version of the Picard
group:

\begin{defn}
The {\bf Picard Lie algebra} of a Poisson manifold $(M,\pi)$ is:
\[
\pic(M):= \frac{\{(Z,\beta)\in\X(M)\ltimes \Omega^2_{\cl}(M):
\Lie_Z\pi=\pi^\sharp(\beta)\}}{\mathfrak{I}},
\]
where $\mathfrak{I}$ is the ideal consisting of pairs $(\pi^\sharp(\eta),\d\eta)$, for $\eta\in\Omega^1(M)$.
\end{defn}

Lemma~\ref{lem:ideal} ensures that $\pic(M)$ indeed has a natural
Lie algebra structure.

\subsection{The Picard Lie algebra exact sequence}

It is essentially a consequence of the definition of the Picard Lie
algebra that it fits into the Lie algebra version of the exact
sequence \eqref{eq:main:exact:seq}. We start by considering the
infinitesimal version of the group morphism
$\Bis(\Sigma(M))\to\G_\pi(M)$:

\begin{lem}
The map
\begin{equation}
\label{eq:bis:gauge:algebra} \Omega^1(M)\to \frakg_\pi(M),\;\;\;
\eta\mapsto (\pi^\sharp(\eta),\d\eta),
\end{equation}
is a homomorphism of Lie algebras.
\end{lem}

\begin{proof}
For any Poisson tensor $\pi$, we know that the contraction operator
$\pi^\sharp:\Omega^1(M)\to \X(M)$ maps the Koszul bracket to the
usual Lie bracket of vector fields. On the other hand, we have for
any 1-forms $\eta_1,\eta_2\in\Omega^1(M)$:
\begin{align*}
\d [\eta_1,\eta_2]_{\pi}&=\d(\Lie_{\pi^\sharp\eta_1}\eta_2-\Lie_{\pi^\sharp\eta_2}\eta_1-\d(\pi(\eta_1,\eta_2)))\\
                                                &=\Lie_{\pi^\sharp\eta_1}\d\eta_2-\Lie_{\pi^\sharp\eta_2}d\eta_1.
\end{align*}
Therefore,
\[ [(\pi^\sharp(\eta_1),\d\eta_1),(\pi^\sharp(\eta_1),\d\eta_1)]_{\frakg_\pi(M)}=(\pi^\sharp([\eta_1,\eta_2]_{\pi}),
\d [\eta_1,\eta_2]_{\pi}).
\]
\end{proof}

The Lie algebra $(\Omega^1(M),[~,~]_\pi)$ has the following Lie
subalgebras:
\begin{itemize}
\item The bracket of closed 1-forms is again a closed 1-from, so
$\Omega^1_\cl(M)\subset \Omega^1(M)$ is a Lie subalgebra. This Lie
algebra can be identified with the \textbf{Lie algebra of the group
$\LBis(\Sigma(M))$} of Lagrangian bisections of $\Sigma(M)$: for any
$\al\in \Omega^1_\cl(M)$, the exponential map $\exp(t\al)$ gives a
1-parameter family of Lagrangian bisections of $\Sigma(M)$
(cf.~Proposition \ref{prop:bisections:form}).

\item $\Omega^1(M)$ contains the abelian subalgebra of basic forms:
\[ \Omega^1_\bas(M)=\{\al\in\Omega^1(M)~:~ i_{X_f}\al=0,~ \Lie_{X_f}\al=0,\forall f\in C^\infty(M)\}. \]
This Lie algebra can be identified with the \textbf{Lie algebra of
the group $\IBis(\Sigma(M))$} of static bisections: if $\al\in
\Omega^1_\bas(M)$,  the exponential map $\exp(t\al)$ gives a
1-parameter family of bisections of $\Sigma(M)$ taking values in the
isotropy groups (see the discussion before Proposition
\ref{prop:bisections:form}).
\end{itemize}
It follows that the closed basic forms
$\Omega^1_{\cl,\bas}(M)\subset \Omega^1(M)$ form a Lie subalgebra.

The map \eqref{eq:bis:gauge:algebra} and the definition of the Picard Lie algebra lead to the Lie algebra version of the exact sequence
\eqref{eq:main:exact:seq}:

\begin{thm}
\label{thm:main:algebra} The map \eqref{eq:bis:gauge:algebra} fits
into an exact sequence of Lie algebras,
\begin{equation}
\label{eq:main:exact:seq:algebra} \xymatrix{ 0\ar[r]
&\Omega^1_{\cl,\bas}(M)\ar[r]&
\Omega^1(M)\ar[r]&\frakg_\pi(M)\ar[r]&\pic(M)\ar[r]&0,}
\end{equation}
where the first arrow is the natural inclusion and the last is the
quotient projection.
\end{thm}

\begin{proof}
All that remains to be shown is exactness at the stage
$\Omega^1(M)$. Note that a 1-form $\eta$ is mapped to zero under the
map $\Omega^1(M)\to\frakg_\pi(M)$ if and only if $\eta$ is closed
and $\pi^\sharp(\eta)=0$. This last condition is equivalent to
$i_{X_f}\eta=0$ for all $f\in C^\infty(M)$, i.e., $\eta$ is basic.

\end{proof}

This theorem leads to an alternative interpretation of $\pic(M)$,
which will be useful later for explicit computations. Recall that
given a morphism of complexes $\Phi:(A^\bullet,\d)\to
(B^\bullet,\d)$, one can introduce a relative complex
$C_\Phi^\bullet:=A^{\bullet+1}\oplus B^\bullet$ with differential
\[
\d (a,b)=(\d a, \Phi(a)-\d b).
\]
Since we have the short exact sequence
\[
\xymatrix{ 0 \ar[r] & (B^\bullet,\d) \ar[r] & (C_\Phi^\bullet,\d)
\ar[r] & (A^{\bullet+1},\d)\ar[r] & 0,}
\]
denoting the cohomology of the relative complex by
$H^\bullet(\Phi)$, we have a long exact sequence
\[
\xymatrix{ \cdots \ar[r] & H^\bullet(B)\ar[r] & H^\bullet(\Phi)
\ar[r] & H^{\bullet+1}(A) \ar[r] & H^{\bullet+1}(B)\ar[r] & \cdots}
\]
When we apply this construction to the morphism of complexes
 \[
\pi^\sharp:(\Omega^\bullet(M),\d_{\mathrm{dR}})\to
(\X^\bullet(M),\d_\pi),
\]
we obtain

\begin{cor}
\label{cor:relative:exact:sequence} The Picard Lie algebra is given
by
 \[
 \pic(M)=H^1(\pi^\sharp).
 \]
In particular, it fits into a long exact sequence:
\[
\xymatrix{ \cdots \ar[r] & H^1(M)\ar[r] & H^1_\pi(M)\ar[r] &\pic(M)\ar[r] & H^2(M)\ar[r] & H^2_\pi(M)\ar[r] &\cdots}
\]
\end{cor}

\begin{proof}
The definitions above show that the pair
$(Z,\beta)\in\Omega^2(M)\times \X^1_\pi(M)$ is a cocycle in the
relative complex if and only if
\[
\d\beta=0, \quad \d_\pi Z=\Lie_Z\pi=\pi^\sharp(\be).
\]
On the other hand, two cocycles $(Z_1,\beta_1)$ and $(Z_2,\beta_2)$
are cohomologous if and only if there exists a pair
$(\eta,h)\in\Omega^1(M)\times \X^0_\pi(M)$ such that
\[
\beta_1-\beta_2=\d\eta,\quad Z_1-Z_2=\pi^\sharp(\eta)-\d_\pi h=\pi^\sharp(\eta)-\pi^\sharp(\d h).
\]
Hence, replacing $\eta$ by $\eta-\d f$, we conclude that
\[
H^1(\pi^\sharp)=\frac{\{(Z,\beta)\in\X(M)\ltimes \Omega^2_{\cl}(M):
\Lie_Z\pi=\pi^\sharp(\beta)\}}{\sim},
\]
where $(Z_1,\beta_1)\sim(Z_1,\beta_2)$ if and only there exists a
1-form $\eta$ such that $Z_1-Z_2=\pi^\sharp\eta$ and
$\beta_1-\beta_2=\d\eta$. This agrees with the definition of
$\pic(M)$.

\end{proof}

Finally, we note that the exact sequence
\eqref{eq:main:exact:seq:algebra} has the following exact
subsequences corresponding to the infinitesimal versions of \eqref{eq:Diff:exact:seq} and
\eqref{eq:Gauge:exact:seq}:

\begin{cor}
There is an exact sequence of Lie algebras,
\begin{equation}
\label{eq:Diff:exact:seq:algebra} \xymatrix{ 0\ar[r]
&\Omega^1_{\cl,\bas}(M)\ar[r]&
\Omega^1_{\cl}(M)\ar[r]&\X_\pi(M)\ar[r]&\pic(M),}
\end{equation}
such that the image of the last homomorphism is the Lie subalgebra
of $\pic(M)$ given by
\[  \frac{H^1_\pi(M)}{H^1(M)}=\frac{\X_\pi(M)}{\pi^\sharp(\Omega^1_\cl(M))}\subset\pic(M), \]
which coincides with the Lie algebra of the group $\OutAut(M)$ of
outer Poisson automorphisms.
\end{cor}

\begin{cor}
There is an exact sequence of Lie algebras,
\begin{equation}
\label{eq:Gauge:exact:seq:algebra} \xymatrix{ 0\ar[r]
&\Omega^1_{\cl,\bas}(M)\ar[r]&
\Omega^1_{\bas}(M)\ar[r]&\Omega^2_{\cl,\bas}(M)\ar[r]&\pic(M),}
\end{equation}
such that the image of the last homomorphism is the Lie subalgebra
of $\pic(M)$ given by
\[  H^2_\bas(M)=\frac{\Omega^2_{\cl,\bas}(M)}{\d \Omega^1_{\bas}(M)}\subset\pic(M), \]
which coincides with the Lie algebra of the group $\OutGaug(M)$ of
outer gauge transformations.
\end{cor}

It follows from these corollaries that we have a Lie subalgebra of the Picard Lie algebra $\pic(M)$ given by the semi-direct product
\[ \frac{H^1_\pi(M)}{H^1(M)}\ltimes H^2_\bas(M),\]
where a cohomology class $[Z]\in H^1_\pi(M)$ acts on a class $[\beta]\in H^2_\bas(M)$ by Lie derivative.
\newpage

\section{Applications and examples}

We will now show how the methods that we have developed in the
previous sections can be used to compute the Picard group.

\subsection{Computing the Picard group}

Theorem \ref{thm:main} and its corollaries have some immediate
applications to the computation of the Picard group. The following
is a consequence of the exact sequence \eqref{eq:main:exact:seq}:

\begin{cor}
\label{cor:Picard:bisections} If $(M,\pi)$ is a Poisson manifold for
which every bimodule admits a bisection, then
\[
 \Pic(M)\simeq \G_\pi(M)/\Bis(\Sigma(M)).
\]
\end{cor}

From the exact sequence \eqref{eq:Diff:exact:seq}, we obtain:

\begin{cor}
\label{cor:Picard:Lagrangian:bisections} If $(M,\pi)$ is a Poisson
manifold for which every bimodule admits a Lagrangian bisection,
then
\[
\Pic(M)\simeq \OutAut(M).
\]
In particular, this happens if $M$ is compact, with $H^2(M)=0$ and
every bimodule admitting a bisection.
\end{cor}

\begin{proof}
If every bimodule admits a Lagrangian bisection, the last morphism
in the sequence \eqref{eq:Diff:exact:seq} is surjective, so the
result follows.

For the second part, assume that every bimodule admits a bisection.
Then Theorem \ref{eq:main:exact:seq} shows that every element of the
Picard group can be represented by a bimodule of the form
$\Sigma(M)_{(\phi,B)}\in\G_\pi(M)$. Since $H^2(M)=0$, we can choose
$\al\in \Omega^1(M)$ a primitive of $(\phi^{-1})^*B$.  Since $M$ is
compact, the bisection $b=\exp(-\al)$ is defined and, by Proposition
\ref{prop:bisections:form}, it satisfies
$b^*\Omega=-\d\al=-(\phi^{-1})^*B$. Hence, $\Sigma(M)_{(\t\circ
b,-(\phi^{-1})^*B)}$ is isomorphic to the trivial bimodule. We
conclude that
\[
[\Sigma(M)_{(\phi,B)}]= [\Sigma(M)_{(\phi,B)}*\Sigma(M)_{(\t\circ
b,-(\phi^{-1})^*B)}]=[\Sigma(M)_{(\phi\circ\t\circ b,0)}].
\]
This shows that every element in $\Pic(M)$ can be represented by an
element of the form $\Sigma(M)_{(\phi,0)}$, so it admits a
Lagrangian bisection (namely, the identity bisection).
\end{proof}

Additionally, from the exact sequence \eqref{eq:Gauge:exact:seq}, we
obtain:

\begin{cor}
\label{cor:Picard:static:bisections} Let $(M,\pi)$ be a Poisson
manifold such that every bimodule admits a static bisection. Then
\[
\Pic(M)\simeq  \OutGaug(M)).
\]
\end{cor}

We also have infinitesimal versions of these results. The most
important and natural one is the following infinitesimal version of
Corollary \ref{cor:Picard:Lagrangian:bisections}:

\begin{cor}
\label{cor:Picard:H1} If $H^2(M)=0$ then
\[
\pic(M)\simeq H^1_\pi(M)/\pi^\sharp(H^1(M)).
\]
In particular, if $H^2(M)=H^1(M)=0$ then the Picard Lie algebra
$\pic(M)$ is isomorphic to the first Poisson cohomology
$H^1_\pi(M)$.
\end{cor}

\begin{proof}
Apply the exact sequence from Corollary
\ref{cor:relative:exact:sequence}.
\end{proof}

\subsection{Symplectic manifolds}\label{subsec:sympexample}

Let $(S,\pi=\omega^{-1})$ be a symplectic manifold, so that
$\pi^\sharp:(\Omega^\bullet(M),\d_\pi)\to(\X^\bullet(M),\d_\pi)$ is
an isomorphism. The following is a direct consequence of Corollary
\ref{cor:relative:exact:sequence}:

\begin{prop}
For a symplectic manifold $(S,\pi=\omega^{-1})$, we have
$\pic(S)=0$.
\end{prop}

This result indicates that $\Pic(S)$ should be a discrete group. In
fact, as we have already observed in Example \ref{ex:sympl}, the
Picard group of $S$ is identified with the group of outer
automorphisms of $\pi_1(S)$:
\[
\Pic(S)\simeq \OutAut(\pi_1(S)).
\]
Recall (see Example~\ref{ex:sympl}) that this identification
associates to a class $[\phi]\in\OutAut(S)$ the class of the
bimodule $\widetilde{S}\times_{\pi_1(S)}\overline{\widetilde{S}}$,
where $\widetilde{S}$ is the universal covering space of $S$ and
$\pi_1(S)$ acts on one factor by deck transformations and on the
other factor via the automorphism $\phi:\pi_1(S)\to \pi_1(S)$.

The image of the homomorphism $\G_\pi(S)\to\Pic(S)$ consists of
those bimodules which admit a bisection, and the image of
$\Aut(S)\to \Pic(S)$ consists of those bimodules which admit a
Lagrangian bisection. The following proposition shows that the
existence of bisections amounts to a Nielsen-type realization
problem:

\begin{prop}
A bimodule $\widetilde{S}\times_{\pi_1(S)}\overline{\widetilde{S}}$
associated with $[\phi]\in\OutAut(\pi_1(S))$ admits a (Lagrangian)
bisection if and only if there is (symplectic) diffeomorphism
$\Phi:S\to S$ such that $\phi=\Phi_*$.
\end{prop}

\begin{proof}
We can realize the bimodule as the quotient
\[
\xymatrix{
\widetilde{S}\times\overline{\widetilde{S}}\ar@<3pt>[d]\ar@<-3pt>[d]\ar[rr]&& \widetilde{S}\times_{\pi_1(S)}\overline{\widetilde{S}}\ar@<3pt>[d]\ar@<-3pt>[d]\\
 \widetilde{S}\ar[rr] \ar@<-4pt>@{-->}@/_/[u]_{\tilde{b}}&& S, \ar@<-4pt>@{-->}@/_/[u]_{b}
}
\]
where $\pi_1(S)$ acts on the product
$\widetilde{S}\times\overline{\widetilde{S}}$ by acting by deck
transformation on one factor and via the automorphism
$\phi:\pi_1(S)\to \pi_1(S)$ on the other.

Given a (symplectic) diffeomorphism $\Phi:S\to S$ such that
$\Phi_*=\phi$, let $\tilde{\Phi}:\widetilde{S}\to \widetilde{S}$ be
a lift of $\Phi$ to the universal covering space. Notice that
\[
\tilde{\Phi}([\gamma]\cdot x)=\phi([\gamma])\tilde{\Phi}(x),\quad
[\gamma]\in\pi_1(S).
\]
Hence, we can define a (Lagrangian) bisection
\[ \tilde{b}:\widetilde{S}\to\widetilde{S}\times\overline{\widetilde{S}},\quad x\mapsto [x,\tilde{\Phi}(x)], \]
which is $\pi_1(M)$-equivariant. So this bisection descends to a
(Lagrangian) bisection
\[
b:S\to \widetilde{S}\times_{\pi_1(S)}\overline{\widetilde{S}},
\]
and one may verify that $\Phi=\t\circ b$.

Conversely, given a (Lagrangian) bisection
\[
b:S\to \widetilde{S}\times_{\pi_1(S)}\overline{\widetilde{S}},
\]
we can lift it to a bisection:
\[
\tilde{b}:\widetilde{S}\to\widetilde{S}\times\overline{\widetilde{S}}.
\]
To see this, observe that the quotient maps in the diagram above are
local diffeomorphisms, so there is an embedded (Lagrangian)
submanifold $N\subset \widetilde{S}\times\widetilde{S}$ which covers
the image of $b$. The restriction of the projections to a connected
component of $N$ are covering maps of $\widetilde{S}$, so they must
be diffeomorphisms. Hence a connected component of $N$ is the graph
of a (Lagrangian) bisection $\tilde{b}$ which lifts the (Lagrangian)
bisection $b$. Now observe that any (Lagrangian) bisection
$\tilde{b}$ must be of the form $x\mapsto [x,\Psi(x)]$, where
$\Psi:\widetilde{S}\to \widetilde{S}$ is a (symplectic)
diffeomorphism satisfying
\[
\Psi([\gamma]\cdot x)=\phi([\gamma])\Psi(x),\quad
[\gamma]\in\pi_1(S).
\]
Since $\tilde{b}$ covers the section $b$, it follows that
$\Psi:\widetilde{S}\to \widetilde{S}$ covers the (symplectic)
diffeomorphism $\Phi:=\t\circ b:S\to S$, and that this map satisfies
\[ \Phi_*([\gamma])=\phi([\gamma]). \]
\end{proof}

We now provide examples of symplectic manifolds showing that,
regarding existence of bisections, all possibilities can occur.

\begin{examp}
If $S$ is a closed oriented surface, then the Dehn-Nielsen-Baer Theorem
(see, e.g., \cite[Chap.~8]{FM12}) shows that every automorphism of $\pi_1(S)$ is
realizable by a diffeomorphism, so that every bimodule admits a
bisection. On the other hand, if one takes, e.g., $S=\Tt^2$ and
considers an automorphism $\phi\in \GL(2,\Zz)=\Aut(\pi_1(\Tt^2))$
with determinant $-1$, then any diffeomorphism $\Phi:\Tt^2\to\Tt^2$
realizing $\phi$ is orientation reversing, so it is never
symplectic. The corresponding bimodule has bisections, but no
Lagrangian bisections.
\end{examp}

\begin{examp}
\label{ex:Gustavo} Let $S$ be the symplectic fibration over $\Tt^2$
with fiber $\Ss^2\times \Ss^2$ obtained from the mapping torus
defined by the symplectomorphisms of $\Ss^2\times \Ss^2$ given by
$\phi_1(x,y)=(x,y)$ and $\phi_2(x,y)=(y,x)$. We can extend the
symplectic structure on the fibers to a symplectic structure on $S$,
since the fibers are 1-connected and the base is symplectic. Now let
\[
\phi=\left(\begin{matrix} k & l \\ r & s
\end{matrix}\right)\in\GL(2;\Zz)=\Aut(\pi_1(S)),
\]
and assume there exists a diffeomorphism $\Phi: S \to S$ inducing
$\phi:\pi_1(S)\to\pi_1(S)$. Then $\Phi$ must be homotopic to a
fibred homotopy equivalence $\Psi:S \to S$. Let $\psi: \Ss^2\times\Ss^2 \to \Ss^2\times\Ss^2$ be
the restriction of $\Psi$ to a fiber. In homotopy we have the
following relations:
\begin{align*}
\psi\phi_1 &= \phi_1^k \phi_2^r \psi,\\
\psi\phi_2 &= \phi_1^l \phi_2^s \psi.
\end{align*}
Since $\phi_1$ is the identity, the first equation shows that
$\phi_2^r$ must be homotopic to the identity. If $r$ is odd, this is
a contradiction.

So any $\phi\in\Aut(\pi_1(S))$ with $r$ odd is not realizable by a
diffeomorphism, and the corresponding bimodule does not admit any
bisection.
\end{examp}

\subsection{Zero Poisson structures}\label{subsec:zero}

Let $(M,\pi\equiv 0)$ be a manifold equipped with the zero Poisson
structure. The following is a direct consequence of Corollary
\ref{cor:relative:exact:sequence}:

\begin{prop}
If $M$ is equipped with the zero Poisson structure, then
\[
\pic(M)=H^1_\pi(M)\ltimes H^2(M) = \mathfrak{X}(M)\ltimes H^2(M).
\]
\end{prop}

We now revisit the Picard group. Let $M\stackrel{p}{\lmap}
(S,\omega)\stackrel{q}{\rmap} \overline{M}$ be a bimodule for the
zero Poisson structure on $M$. Since the symplectic groupoid
$\Sigma(M)=T^*M$ has $\s=\t$, the fibers of $p$ and $q$ must
coincide, so that $p=\phi\circ q$ for some diffeomorphism $\phi:M\to
M$. Since the fibers of $\s=\t$ are contractible, so are the fibers
of $p$ and $q$. It follows that we can choose a section $b:M\to S$
of $q:S\to M$, which is automatically a bisection and $p\circ
b=\phi$. Hence, every bimodule has a bisection and, by Corollary
\ref{cor:Picard:bisections},
\[
\Pic(M)\simeq \G_\pi(M)/\Bis(T^*M).
\]
Now observe that
\begin{enumerate}
\item[(i)] $\G_\pi(M)=\Diff(M)\ltimes \Omega_\cl^2(M)$, since $\pi=0$;
\item[(ii)] $\Bis(T^*M)=\Omega^1(M)$ and the action of $\al\in\Omega^1(M)$ on a
pair $(\phi,B)$ gives $(\phi,B+\d\al)$.
\end{enumerate}
It follows that
\[
\Pic(M)\simeq \Diff(M)\ltimes H^2(M),
\]
which gives another proof of the result of  \cite{BW04} as a
consequence of more general principles revealed by our methods.

\subsection{Linear Poisson structures}

In sections \ref{subsec:sympexample} and \ref{subsec:zero} we saw
how our general method recovers results from \cite{BW04}. We now
discuss the Picard group of a linear Poisson structure on the dual
of a Lie algebra $\gg$, an example proposed in \cite{BW04} but not
handled there. We start with a description of the Picard Lie
algebra:

\begin{thm}
For the linear Poisson structure on the dual of a Lie algebra $\gg$,
the Picard Lie algebra agrees with the 1st Poisson cohomology:
\[
\pic(\gg^*)\simeq H^1_\pi(\gg^*).
\]
In particular, if $\gg$ is compact and semi-simple, then
$\pic(\gg^*)=0$.
\end{thm}

\begin{proof}
For the first assertion, apply Corollary \ref{cor:Picard:H1}. For
the second part, we recall that if $\gg$ is a compact semi-simple
Lie algebra then $H^1_\pi(\gg^*)=0$. In fact, it is proved in
\cite{GW92} that, for a compact semi-simple Lie algebra,
\[
 H_\pi^\bullet(\gg^*)\simeq H^\bullet(\gg)\otimes \mathrm{Cas}(\gg^*).
 \]
By the first Whitehead Lemma, $H^1(\gg)=0$, so the lemma follows.
One can also prove this by observing that if $G$ is the 1-connected Lie group integrating $\gg$, then
$\gg^*$ integrates to the proper groupoid $T^*G\tto \gg^*$, which is source 1-connected. The
Van Est Theorem of \cite{Cr03} shows that the 1st differentiable
groupoid cohomology of $T^*G$ is isomorphic to the 1st Poisson
cohomology of $\gg^*$. The vanishing of differentiable cohomology
for proper Lie groupoids \cite{Cr03} implies that
$H^1_\pi(\gg^*)=0$.
\end{proof}

This theorem leads to a natural conjecture: the Picard group of
$\gg^*$ is isomorphic to the group of outer Poisson automorphisms of
$\gg^*$,
\[
\Pic(\gg^*)\simeq\OutAut(\gg^*).
\]
Although we cannot prove this in general, we can show that this
result holds for a compact, semi-simple Lie algebra. In fact, we
have the following result, which was conjectured in \cite{BW04}:

\begin{thm}
\label{thm:compact:semisimple} If $\gg$ is a compact semi-simple Lie
algebra then
\[ \Pic(\gg^*)\simeq \OutAut(\gg^*)\simeq \OutAut(\gg). \]
\end{thm}

The proof of this theorem is a direct consequence of the next two
propositions.


\begin{prop}
\label{prop:compact:type:1} If $\gg$ is a compact semi-simple Lie algebra then every bimodule $(S,\omega)\tto \gg^*$ has a bisection.
\end{prop}

We defer the proof of this proposition to the end of the section.
Since $\gg^*$ is a vector space, this lemma allows us to invoke
Corollary \ref{cor:Picard:bisections} to conclude that
\[
\Pic(\gg^*)\simeq \OutAut(\gg^*).
\]
We now conclude the proof of Theorem \ref{thm:compact:semisimple}
with

\begin{prop}
\label{prop:compact:type:2} If $\gg$ is a compact semi-simple Lie
algebra then
\[
\OutAut(\gg^*)\simeq \OutAut(\gg).
\]
\end{prop}

\begin{proof}
Let $\phi:\gg^*\to\gg^*$ be a Poisson diffeomorphism. Since $\phi$
maps symplectic leaves to symplectic leaves and, by semi-simplicity,
$\{0\}$ is the only zero dimensional leaf, we must have $\phi(0)=0$.
Since the Poisson structure on $\gg^*$ is already linear, we
conclude that $\d_0\phi:\gg^*\to\gg^*$ is a linear Poisson
isomorphism, which is equivalent to the inverse transpose
$(\d_0\phi^{-1})^*:\gg\to\gg$ being a Lie algebra automorphism. In
this way, we obtain a group homomorphism
\[
\Aut(\gg^*)\to \Aut(\gg),\ \phi\longmapsto
\phi_*\equiv(\d_0\phi^{-1})^*.
\]
If $b\in\LBis(\Sigma(\gg^*))=\LBis(T^*G)$ is a lagrangian bisection
inducing a inner automorphism $\phi$, then we find that the induced
automorphism of $\gg$ is inner:
\[ \phi_*(v)=\Ad g\cdot v, \]
where $g=b(0)$. It follows that we have  a well-defined group
homomorphism
\[
\OutAut(\gg^*)\to \OutAut(\gg),\ [\phi]\longmapsto [\phi_*].
\]
We claim that this is a group isomorphism:
\begin{enumerate}[(i)]
\item \emph{Surjectivity}:
Given $[l]\in \OutAut(\gg)$, with $l:\gg\to\gg$ a Lie algebra
automorphism, the map $\phi=(l^{-1})^*:\gg^*\to\gg^*$ is a Poisson
automorphism and $[\phi_*]=l$.

\item \emph{Injectivity}:
Assume that $\phi:\gg^*\to\gg^*$ is a Poisson diffeomorphism such
that $\phi_*$ is a inner automorphism. We need to show that $\phi$
is inner.

We know that $(\d_0\phi^{-1})^*=\Ad g$, for some $g\in G$. Since for
a compact Lie group $\exp:\gg\to G$ is a surjective map, there
exists $v\in\gg$ such that $\exp(v)=g$. If we let $f_v(\xi)=\langle
\xi,v\rangle$, then the Lagrangian bisection $\exp(\d f_v)$ defines
a inner automorphism $\psi:\gg^*\to\gg^*$ such that $\psi_*=\phi_*$.
Hence, after composing $\phi$ with $\psi^{-1}$, we can assume that
$\phi_*=$id. We will use  a Moser-type trick to show that $\phi$ is
inner.

Let $\phi_t:\gg^*\to\gg^*$ be the Poisson isotopy from $\phi$ to the
identity given by
\[
\phi_t(x)=\begin{cases} \frac{1}{t}\phi(tx), &\mbox{if } t\not=0, \\
x, & \mbox{if } t=0. \end{cases}
\]
Since each $\phi_t$ preserves $\pi$, the corresponding vector
field $X_t(x)\equiv \frac{\d}{\d t}\phi_t(x)$ is a time-dependent
Poisson vector field:
\[
\Lie_{X_t}\pi=0.
\]
The condition $H^1_\pi(\gg^*)=0$ implies that, for each $t$, there
is a function $f_t\in C^\infty(\gg^*)$ such that $X_t=\pi^\sharp \d
f_t$. Moreover, we can assume that the family $f_t$ depends smoothly
on the parameter $t$ (see \cite[pp. 449]{GW92}, Remark 1). Then
$b_t=\exp(\d f_t)$ defines a 1-parameter family of Lagrangian
bisections such that the corresponding inner Poisson automorphisms
coincide with $\phi_t$. In particular, $\phi=\phi_1$ is inner.
\end{enumerate}
\end{proof}

To finish the proof of Theorem \ref{thm:compact:semisimple}, we
present the proof of Proposition \ref{prop:compact:type:1}:

\begin{proof}[Proof of Proposition \ref{prop:compact:type:1}]
Given a  bimodule $(S,\omega)\tto \gg^*$, we must show that it has a
bisection. Note that the canonical integration of $\gg^*$ is the
action groupoid $G\times \gg^*\tto\gg^*$ associated with the
coadjoint action of the compact, 1-connected, Lie group $G$
integrating $\gg$. Therefore, a bimodule amounts to two commuting
Hamiltonian free actions of $G$ on $S$ with moment maps
$p,q:S\to\gg^*$, such that each moment map is the quotient map to
the orbit space of the other action.

Observe that the fibers of $p$ and $q$ are diffeomorphic to $G$.
Since $H^1(G)=0$ and $H^2(G)=0$, the same holds for $q$- and
$p$-fibers. It follows that $H^2(S)=0$, so $\omega$ is exact.
Noticing that $p^{-1}(0)=q^{-1}(0)$ is Lagrangian (and has vanishing
first cohomology), we see that we can find a primitive $\alpha$ for
$\omega$, $\omega=\d\al$, such that $\al|_x=0$ for $x\in
p^{-1}(0)=q^{-1}(0)$. By averaging over the two commuting actions,
we can choose additionally $\al$ to be $G$-invariant.

It follows that the unique vector field $Y\in\X(S)$ satisfying
\[
i_Y\omega=\al,
\]
is $G$-invariant under both actions, and $Y|_x=0$ if $x\in
p^{-1}(0)=q^{-1}(0)$. Note that
\[
\Lie_Y\omega=\omega.
\]
Let $Y_1=p_*Y$ and $Y_2=q_*Y$ be the projections of $Y$ on $\gg^*$.
Since $p$ and $q$ are Poisson/anti-Poisson maps, it follows that
\[
\Lie_{Y_i}\pi=\pi, \;\;\; (i=1,2).
\]
Since $\pi$ is a linear Poisson structure, the Euler vector field
$E=\sum_{i=1}^d \xi_i\frac{\partial}{\partial \xi_i}\in\X(\gg^*)$
also satisfies $\Lie_E\pi=\pi$.
Hence $Y_i-E$, $i=1,2$, are Poisson vector fields. The condition
$H^1_\pi(\gg^*)=0$ implies that there exist smooth functions $h_i\in
C^\infty(\gg^*)$ such that
\[
Y_i=E+\pi^\sharp(\d h_i).
\]
By setting $\tilde{E}:=Y-X_{h_1\circ p}-X_{h_2\circ q}\in\X(S)$, we
obtain a vector field in $S$ satisfying
\[
\Lie_{\tilde{E}}\omega=\omega,\quad p_*(\tilde{E})=q_*(\tilde{E})=E.
\]
Note that $\tilde{E}$ vanishes only at points $x\in
p^{-1}(0)=q^{-1}(0)$ because it projects on $E$, which vanishes only
at the origin. Moreover, $\tilde{E}$ is complete because $E$ is a
complete vector field and the fibers of $p$ (or $q$) are compact. In
a neighborhood of any $x\in p^{-1}(0)=q^{-1}(0)$ we can split
$\tilde{E}$ as a sum of $E$ and a vector field along the compact
fibers of $p$ which vanishes at $p^{-1}(0)$. We conclude that for
any $x\in S$ $\lim_{t\to -\infty}\phi^t_{\tilde{E}}(x)$ exists and
belongs to $G=p^{-1}(0)=q^{-1}(0)$.

Since $\tau:S\to G$, $x\mapsto \lim_{t\to
-\infty}\phi^t_{\tilde{E}}(x)$, is a projection and the
linearization of $\tilde{E}$ at $x\in G$ is also a linear
projection, it follows (see, e.g., \cite{GR09}) that
$S$ is a vector bundle over $G=p^{-1}(0)=q^{-1}(0)$ with projection
$\tau:S\to G$ and Euler vector field $\tilde{E}$. A fiber of $\tau$
gives the desired bisection of $S\tto \gg^*$.
\end{proof}

\subsection{Strong-proper Poisson structures}

A Poisson structure $(M,\pi)$ is called \textbf{strong-proper} if
the symplectic groupoid $\Sigma(M)$ is a proper Hausdorff Lie
groupoid. Examples include any symplectic manifold $S$ with
finite fundamental group or the dual $\gg^*$ of a compact
semi-simple Lie algebra. This class of Poisson manifolds plays the
role of the ``compact objects'' in Poisson geometry and is studied in
detail in \cite{CrFeMa15}. Our methods yield the following result:

\begin{prop}
If $(M,\pi)$ is a strong-proper Poisson manifold then $\pic(M)$ is
a subalgebra of the abelian Lie algebra $H^2(M)$. In particular, if $H^2(M)$ is
finite dimensional, then $\Pic(M)$ is a finite dimensional Lie group with
abelian Lie algebra.
\end{prop}

\begin{proof}
The differentiable groupoid cohomology of a proper Lie groupoid
vanishes, and the Van Est map relating differentiable groupoid
cohomology and Lie algebroid cohomology is an isomorphism in degree
1 (\cite{Cr03}). Hence, in our case, we have
$H^1_\pi(M)=H^1_{\d}(\Sigma(M))=\{0\}$. Therefore in the long exact
sequence of Corollary \ref{cor:relative:exact:sequence} we obtain an
injective homomorphism $\pic(M)\hookrightarrow H^2(M)$.
\end{proof}

Although in all the examples that we have seen of strong-proper
Poisson manifolds we obtained $\pic(M)=0$, the regular
(non-symplectic) examples in \cite{CrFeMa15} have $\pic(M)\not=0$.


\end{document}